\documentclass{amsart}
\usepackage{amsfonts,mathrsfs}
\usepackage{hyperref}
\usepackage{url}
\usepackage{graphicx}
\usepackage{epstopdf}
\usepackage{setspace,epsfig}

\theoremstyle{plain}

\newtheorem{corollary}{Corollary}[section]

\newtheorem{definition}{Definition}[section]

\newtheorem{lemma}{Lemma}[section]

\newtheorem{proposition}{Proposition}[section]

\newtheorem{theorem}{Theorem}[section]
\numberwithin{equation}{section}

\DeclareMathOperator{\Dom}{Dom}
\date{\today}

\begin{document}
\title[Quantum mechanics on Laakso spaces]{Quantum mechanics on Laakso spaces}

\author[C.~Kauffman]{Christopher J.~Kauffman}
\address[C.~Kauffman]{Department of Mathematics, University of Rochester, Rochester, NY  14627, USA}
\email[C.~Kauffman]{ckauffma@u.rochester.edu}

\author[R.~Kesler]{Robert M.~Kesler}
\address[R.~Kesler]{Department of Mathematics, Princeton University, Princeton, NJ  08544, USA}
\email[R.~Kesler]{rkesler@princeton.edu}

\author[A.~Parshall]{Amanda G.~Parshall}
\address[A.~Parshall]{Department of Mathematics, York College of Pennsylvania, York, PA  17403-3651, USA}
\email[A.~Parshall]{aparshal@ycp.edu}

\author[E.~Stamey]{Evelyn A.~Stamey}
\address[E.~Stamey]{Department of Mathematics, Ithaca College, Ithaca, NY  14850, USA}
\email[E.~Stamey]{estamey1@ithaca.edu}

\author[B.~Steinhurst]{Benjamin A.~Steinhurst}
\address[B.~Steinhurst]{Department of Mathematics, Cornell University, Ithaca, NY 14853-4201, USA}
\email[B.~Steinhurst]{steinhurst@math.cornell.edu}
\urladdr{\url{http://www.math.cornell.edu/~steinhurst/}}
\thanks{The authors are supported by the NSF through grant DMS-0505622.}

%\thanks{This paper is in final form and no version of it will be submitted for publication elsewhere.}

\begin{abstract}
We first review the spectrum of the Laplacian operator on a general Laakso Space before considering modified Hamiltonians for the infinite square well, parabola, and Coulomb potentials. Additionally, we compute the spectrum for the Laplacian and its multiplicities when certain regions of a Laakso space are compressed or stretched and calculate the Casimir force experienced by two uncharged conducting plates by imposing physically relevant boundary conditions and then analytically regularizing the result. Lastly, we derive a general formula for the spectral zeta function and its derivative for Laakso spaces with strict self-similar structure before listing explicit spectral values for cases of interest.  

\end{abstract}

\maketitle

%{\small \tableofcontents}
\section{Introduction}

Laakso spaces are introduced in \cite{Laakso2000} as a quotient space of the cartesian product of the unit interval with the middle thirds Cantor set, and in \cite{ST08} it is shown that such spaces can be constructed as the projective limit of quantum graphs, verifying a comment in \cite{BarlowEvans2004}. The motivation behind the projective limit is the construction of Markov processes and associated infinitesimal generators; moreover, the spectrum of the Laplacian generating the most natural of these processes is given in \cite{ST08} by using the quantum graph approximations to construct a complete set of eigenfunctions. Here, we build on these previous results by performing calculations motivated by analogy to Quantum Mechanics in the setting of Laakso spaces. Other authors have investigated similar questions on finitely ramified fractals in \cite{FKS09,S09,BajorinEtAl2008} and more directly in \cite{ADT2009,ADT2010}.

After a preliminary discussion in Section \ref{PB}, which recalls the definition of Laakso spaces and the derivation of the Laplacian, $\Delta_L$, the body of the paper falls into three parts. The first is Section \ref{sect: SA} in which the spectrum of a Hamiltonian operator of the form $H=\Delta +V(x)$ is studied numerically in three cases (infinite square well, parabolic well, and Coulomb) and analytically in the infinite square well case, as well. Then in Section \ref{sect:CasimirForce}, we introduce analytic regularization techniques used in physics literature and Number Theory as a tool to calculate an analogue to the Casimir effect \cite{Casimir1948}, which describes a Quantum Mechanical force experienced by perfectly conducting uncharged plates. Then in Section \ref{sect: Zeta} the spectral zeta function associated to the Laplacian $\Delta_L$ is studied in greater detail, and, in circumstances where the Laakso space is strictly self-similar, specific values are calculated directly. 

\section{Preliminary Background} \label{PB}

\subsection{Laakso Spaces}

Laakso spaces are defined in \cite{Laakso2000} as a quotient space of $I \times K$ by an iteratively defined series of identifications, where \emph{K} is the Cantor set, $I$ is the unit interval,  and $\iota: I \times K \rightarrow L$ is the quotient map. Such spaces are also constructable as projective limits of finite quantum graphs $\{F_{n}\}$, as in \cite{ST08}. The projective limit construction provides a convenient  approximation of any Laakso space.  According to \cite{ST08}, the construction of $F_{n}$ is specified by an equivalence relation that is encoded by a sequence of integers $\{j_{n}\}_{n=1}^\infty$, which give the number of identifications---or subdivisions of each cell---at the $n^{th}$ level of construction.

\begin{figure}[htb]
\begin{center}
\epsfig{file=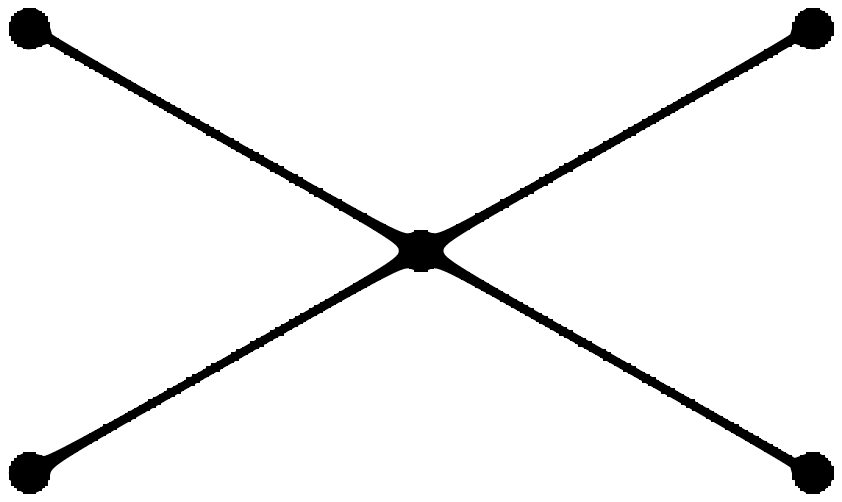, width=3.5cm}\hspace{1cm}
\epsfig{file=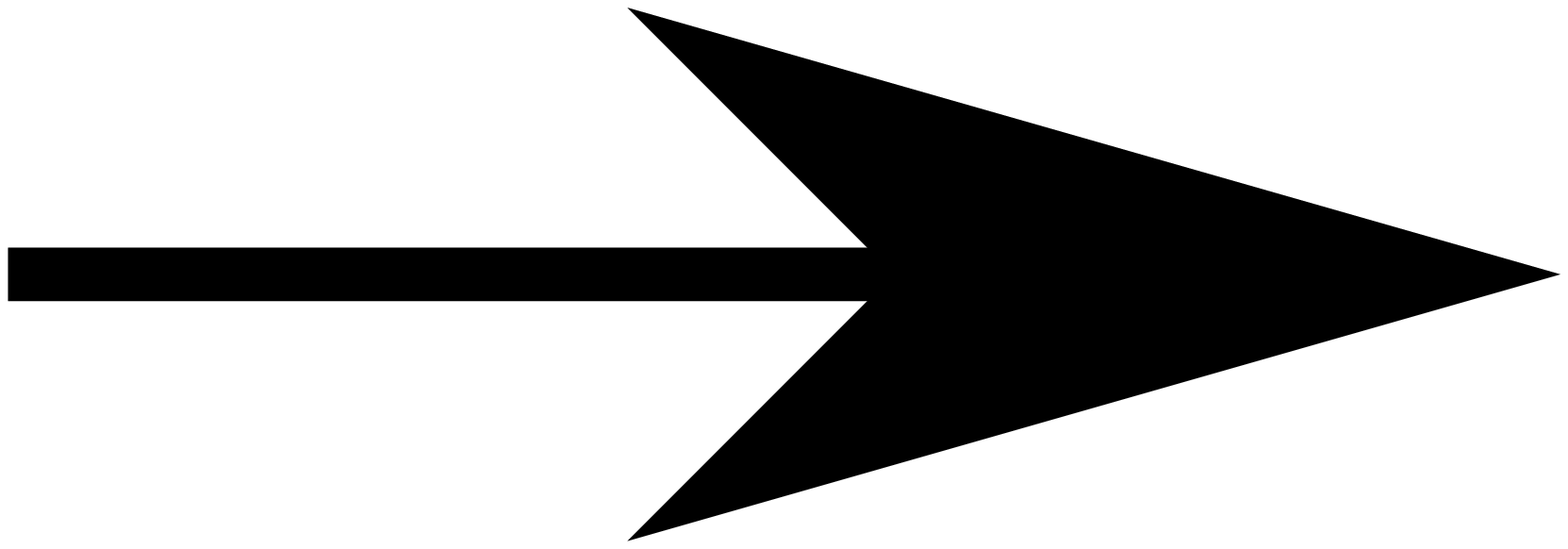, width=1.5cm}\hspace{1cm}
\epsfig{file=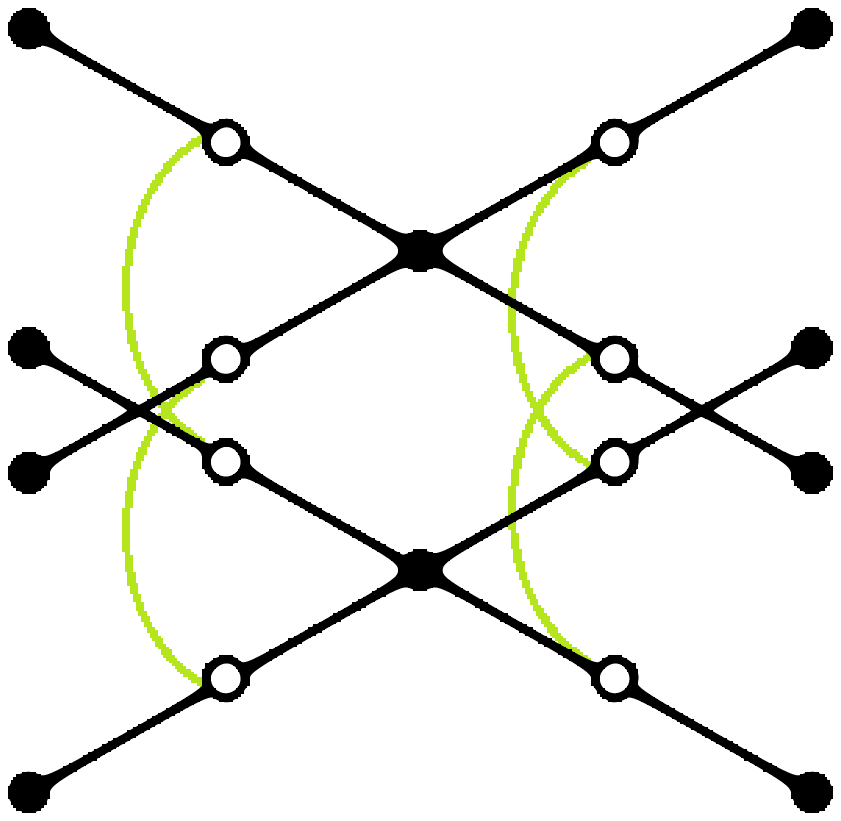, width=3.5cm}\vspace{1cm}
\end{center}
\end{figure}
\vspace{-1.2cm}
\hspace{2cm} $F_{1} \hspace{6.1cm} F_{2}$

\begin{figure}[htb]
\begin{center}
\epsfig{file=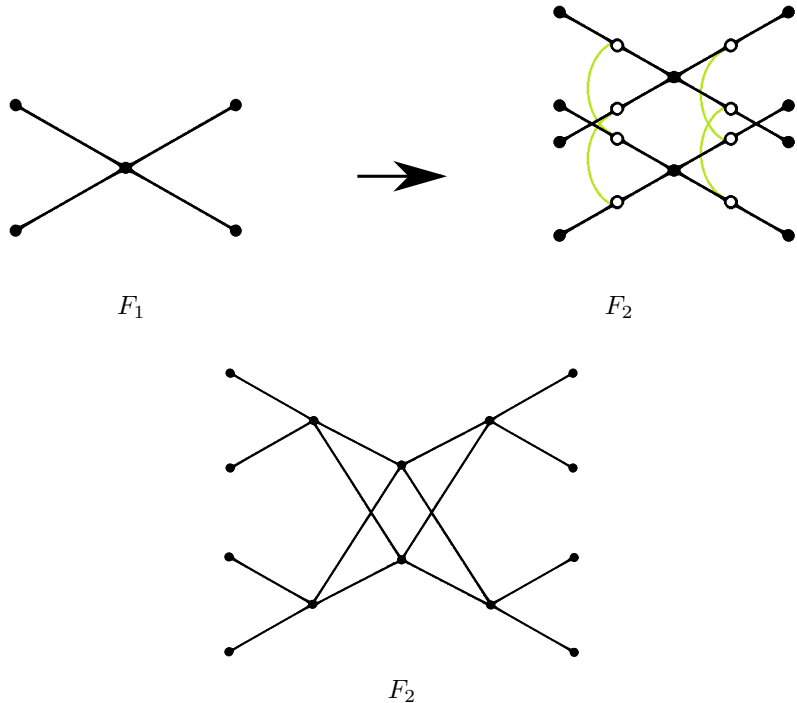, width=5cm}\\
$F_{2}$
\end{center}
\caption[projective limit construction]{Construction of $F_{2}$ from $F_{1}$ with $j_{1}=j_{2}=2$.}
\label{projective limit construction}
\end{figure}

The most elementary approximation is $F_{0}$=$[0,1]$.  We construct $F_{n+1}$ by dividing each interval in $F_{n}$ into $j_{n}$ equal subintervals, the new nodes identify the boundaries of these newly-formed subintervals.  Next, duplicate $F_{n}$ and connect each new node to the corresponding node in the other copy.  For convenience in the counting arguments to come, align the nodes in columns. This ensures that each $F_{n}$ is vertically and horizontally symmetric. The following quantity will be used frequently:
\begin{equation}
I_{n}=\prod_{i=1}^n j_{i},
\end{equation}
where $I_{0}=1$.   (See Figure \ref{projective limit construction}.) The sequence of quantum graphs $\{F_{n}\}$, $n \geq 0$, approximates a specific Laakso space, where the depth of approximation increases as \emph{n} increases. A Laakso space has other important properties which can be expressed in terms of the sequence $\{ j_{n} \}$.  For example, the Hausdorff dimension of a Laakso space is
\begin{equation}
Q_{L}=\lim_{n \rightarrow \infty}\left(1+\frac{\log(2^n)}{\log(I_n)} \right),
\label{Hdimension}
\end{equation}
provided the limit exists.  If $\{j_i\}_{i=1}^{\infty}$ is a repeating sequence with period \emph{T}, equation \ref{Hdimension} conveniently reduces to
\begin{equation}
Q_{L}=1+\frac{\log(2^T)}{\log(I_T)}.
\end{equation}
Furthermore, each $F_{n}$ can be decomposed into three distinct shapes---V's, loops, and crosses (see Figure \ref{A V, a loop, and a cross})---the counts of which are necessary in determining the spectrum of the square well Hamiltonian in Subsection \ref{subsect: SW}  and an arrangement of conducting plates in Subsection \ref{sect:CasimirForce}.

\begin{figure}[htb]
\begin{center}
\epsfig{file=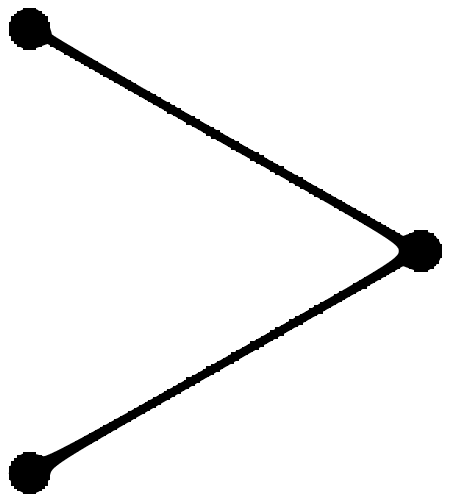, width=1.7cm} \hspace{1.5cm}
\epsfig{file=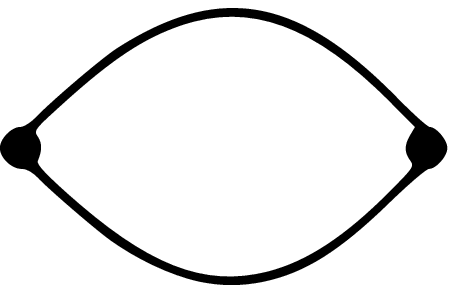, width=1.75cm} \hspace{1.5cm}
\epsfig{file=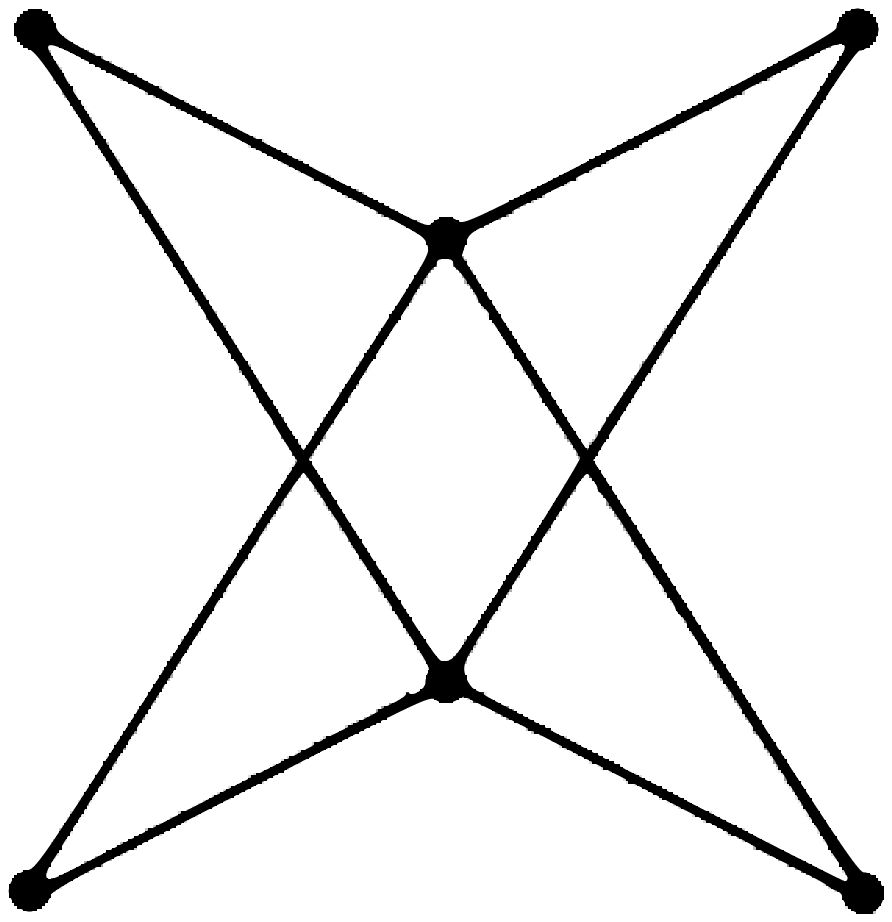, width=3.2cm}
\end{center}
\caption[A V, a loop, and a cross]{A V, a loop, and a cross}
\label{A V, a loop, and a cross}
\end{figure}

\subsection{Laplacian}
In \cite{ST08} the Laplacian $\Delta$ is constructed on a Laakso space as the minimal self-adjoint extension of a compatible sequence of operators $\left\{ A_{n} \right\}$, where each $A_{n}$ acts by $-\frac{d^{2}}{dx_{e}^{2}}$ along edges in the $F_{n}$ quantum graph approximation. Specifically, $\Delta$ acts by
\begin{equation}
\iota^{*}\Delta[f]=\left( -\frac{d^{2}}{dx^{2}} \right) \iota^{*}[f]~\text{where}~x \in \emph{I}~and~\iota: \emph{I} \times \emph{K} \rightarrow \emph{L}~ \text{is the quotient map.} 
\end{equation}
Moreover,  the spectrum of this self-adjoint operator can be decomposed into the union of eigenvalues in an orthogonal basis on each quantum graph, and in \cite{BDMS10} the eigenvalues of $\Delta$ on a Laakso space with associated sequence $\{ j_{n} \}$ are explicitly shown to be 
\begin{eqnarray}  \label{eq: Lspectrum}
\sigma(\Delta) = &\bigcup_{k=0}^\infty \left\{ \pi^{2} k^{2} \right\} \cup \bigcup_{n=1}^\infty \bigcup_{k=0}^\infty \left\{ (k+1/2)^{2} \pi^{2} I_{n}^{2} \right\}  \cup \bigcup_{n=1}^\infty \bigcup_{k=1}^\infty \left\{k^{2} \pi^{2} I_{n}^{2} \right\}  \notag \\
 & \cup \bigcup_{n=2}^\infty \bigcup_{k=1}^\infty \left\{ k^{2} \pi^{2} I_{n}^{2} \right\} \cup \bigcup_{n=2}^\infty \bigcup_{k=1}^\infty \left\{ \frac{k^{2}\pi^{2} I_{n}^{2}}{4} \right\} 
\end{eqnarray} 
with respective multiplicities:
\begin{equation*} 
1, 2^{n}, 2^{n-1}(j_{n}-2)I_{n-1}, 2^{n-1}(I_{n-1}-1), 2^{n-2}(I_{n-1}-1).
\end{equation*}

\subsection{Quantum Mechanics}
The mathematical formalism of Quantum Mechanics consists of two fundamental types of objects: (1) normalized state vectors commonly denoted by $\mid \psi \rangle$, which represent physical systems and reside in a complex separable Hilbert space; and (2) self-adjoint  operators, which act on state vectors, represent various observable quantities, and whose spectra correspond to the possible values of the measurement of those observables \cite{Griffiths2005}. For example, the position and momentum operators act by $\hat{x}[f]=xf$ and $\hat{p}[f]=-i \hbar \frac{ d}{dx}f$, respectively. In accordance with classical mechanics, we formulate energy as a function of momentum and position, namely $E=\frac{p^{2}}{2m}+V(x)$, and then associate to each observable its corresponding operator. In this way, we motivate the Hamiltonian energy operator $\hat{H}$, which acts by 
\begin{equation}
\hat{H}[f]=\left(\frac{p^{2}}{2m}+V(x,t) \right)[f]=\left( -\frac{\hbar^{2}}{2m}\frac{d^{2}}{dx^{2}} +V(x,t) \right)[f].
\end{equation}
Despite the fact that the domain of each self-adjoint operator is a complex separable Hilbert space, the domains for different operators are, in general, distinct. In the case of the position and momentum spaces, the appropriate domain is $L^{2}(\mathbb{R}^{n})$ with natural  inner product $\langle \phi \mid \psi \rangle=\int_{\mathbb{R}^{n}} \phi\psi^{*} ~dx$ and normalization condition $\langle \psi \mid \psi \rangle=\int_{\mathbb{R}^{n}} \left| \psi \right| ^{2} ~dx=1$. In fact, there is a probabilistic interpretation to the inner product on the position and momentum spaces. If $\mid \psi \rangle$ represents a particle in position or momentum space, then the probability that a measurement of that particle's position or momentum will fall in the interval $[A,B]$ is $ \int_{A}^{B} \left| \psi \right| ^{2} ~dx$ \cite{Griffiths2005}. The normalization condition then arises out of the necessity that a measurement of either the particle's position or momentum--- but not both simultaneously--- will be observed to have some well-defined value. Additionally, if $\left\{ \psi_{n} \right\} $ is an orthonormal basis of eigenfunctions with associated eigenvalues $\left\{\lambda_{n} \right\}$ for the self-adjoint operator $\emph{A}$ with corresponding observable $O_{A}$, then any normalized state vector $\mid \psi\rangle  \in \Dom{(A)}$ can be written as a linear combination of orthonormal eigenfunctions, i.e. $\mid \psi \rangle=\sum_{i=1}^\infty a_{i} \psi_{i}$ with the following probabilistic interpretation: a measurement of $O_{A}$ for the system represented by $\mid\psi \rangle$ will yield the value $\lambda_{n}$ with probability $a_{n}^{2}$. If a particular eigenvalue of $\emph{A}$ is degenerate, then the total probability of observing that eigenvalue is given by adding all associated $a_{i}^{2}$. It follows from these considerations that the expected value of the observable $O_{A}$ is 
\begin{equation}
\langle A \rangle= \langle \psi \mid A \mid \psi \rangle.
\end{equation} 

In keeping with this discussion, we interpret the eigenvalues associated with the infinite square well, parabolic, and Coulomb potentials in Subsections \ref{subsect: SW} and \ref{subsect: CP} as the set of allowable energy measurements for a particle affected by those potentials. In Section \ref{sect:CasimirForce}, eigenvalues represent the permissible energy states for eigenfunctions in the presence of conducting plates.

\section{Spectral Analysis of Hamiltonians} \label{sect: SA}
In this section, we examine the spectrum of three different Hamiltonians. Theorem \ref{thrm: SWspectrum} gives the spectrum and associated multiplicities of the Laplacian with an infinite square well potential. The following two subsections consider the Laplacian with a parabolic potential and a coulomb potential, respectively. In the last subsection, we discuss the numerical approximations of the spectra accompanied by some data.

\subsection{Infinite Square Well Potential} \label{subsect: SW}
Let $x$ be the coordinate on $L$ and $F_n$ depending solely on the ``horizontal'' direction. The function $V(x)$ on $L$ is described so that $\iota^{*}(V)(x,w)$ depends only on $x$ and not on $w$. We first discuss the infinite square well Hamiltonian $H_{SW}$, where

\begin{equation} 
V_{SW}(x) = \left\{ 
 \begin{array}{rl} 
   \infty & : x \hspace{1 mm} \in \hspace{1 mm}[0, \frac{1}{4}) \cup 
(\frac{3}{4}, 1]\\ 
   0 & : x \hspace{1 mm} \in \hspace{1 mm}[\frac{1}{4}, \frac{3}{4}],\\ 
 \end{array} 
\right. 
\end{equation} 
and 
\begin{equation}
H_{SW}[f]=\left( \Delta + V_{SW} \right)[f].
\end{equation}

\begin{definition}
The differential operator $A_{n}$ acts on $F_{n}$ by 
\begin{equation*}
A_{n}[f]= \left(-\frac{d^{2}}{dx_{e}^{2}}+V_{SW}\right) [f]
\end{equation*}
along each edge e with $\Dom(A_{n})=\left\{ \right. f \in C(F_n) \mid f \in H^{2}(e)~\forall e$, except at the square well boundary where f may have a discontinuous derivative, and $f(x)=0~\forall x$ outside of the square well  $\left. \right\}$. \\
\end{definition}

\begin{theorem}
\label{SpectrumOrtho}
Let $\Phi_{n}$ be the projection of the Laakso space onto $F_{n}$ and $D_{n}=\left\{ f \circ  \Phi_{n}| f \in \Dom(A_{n}) \right\}$. Define $D_{0}^{\prime}=D_{0}=\Phi_{0}^{*} \Dom(A_{0})$ and $D_{n}^{\prime}=D_{n-1}^{\prime \perp} \cap D_{n}$. Since $(H_{SW}, \Dom(H_{SW}))$ is the minimal self-adjoint extension of the projective system  $\left( A_{n}, \Dom(A_{n})\right) $, 
\begin{equation*}
\sigma(H_{SW})=\bigcup_{n=0}^\infty \sigma(A_{n}|_{D_{n}^{\prime}}).
\end{equation*}
\end{theorem}

\begin{proof}
The proof of this theorem follows closely from the free case proven in \cite{ST08}.
\end{proof}

Theorem  \ref{thrm: SWspectrum} gives the spectrum and associated multiplicites of this Hamiltonian, which we prove in the remaining portion of this subsection. 

\begin{definition}
The expression $w_{n}=\frac{1}{4}(I_{n})$ denotes the number of columns between $x=0$ and $x=\frac{1}{4}$. Let $d_{n}$ be the x-distance from the wall of the square well to the nearest column of nodes inside the square well; this is well defined by the symmetry of $L$.
\end{definition}

To distinguish one set of loops from another in an arbitrary row, we assign each set a number $m=\{1, 2, \cdots, I_{n-1} \}$, counting from left to right.  Similarly, we distinguish one cross in an arbitrary row by assigning each a number $l=\{1, 2, \cdots, I_{n-1}-1 \}$.

\begin{theorem} \label{thrm: SWspectrum} 
Given any Laakso space, $L$, with associated sequence $\{j_i\}$, the spectrum of $H_{SW}$, $\sigma(H_{SW})$, is

\begin{eqnarray}\small 
\bigcup_{k=1}^{\infty} \{4 \pi^2 k^2\}\cup \bigcup_{k=1}^{\infty} \left\{ \frac{ k^{2} \pi^{2}} {d_{1}^{2}} \right\} \cup \bigcup_{k=1}^{\infty} \left\{ 9 k^{2} \pi^{2} \right\} \cup
\bigcup_{n=1}^\infty \bigcup_{k=1}^{\infty} \left\{  \frac{k^2  \pi^2}{d_{n}^{2}}   \right\}  \cup\bigcup_{n=1}^\infty 
\bigcup_{k=1}^\infty \{k^2\pi^2I_n^2\} \notag\\ \cup \bigcup_{n=2}^{\infty} \bigcup_{k=1}^{\infty}  \left\{ \frac{ k^{2} \pi^{2}}{d_{n}^{2}} \right\} \cup \bigcup_{n=2}^{\infty} \bigcup_{k=1}^{\infty} \left\{ k^{2} \pi^{2} I_{n}^{2} \right\} \cup \bigcup_{n=2}^{\infty} \bigcup_{k=1}^\infty  \left\{ \frac{ k^{2} \pi^{2}}{(d_{n}+\frac{1}{I_{n}})^{2}} 
 \right\}  \notag\\  \cup \bigcup_{n=2}^{\infty} \bigcup_{k=1}^\infty \{k^2\pi^2I_n^2\} \cup \bigcup_{n=2}^\infty 
\bigcup_{k=1}^{\infty} \left\{ \frac{k^2\pi^2 I_n^2}{4} \right\} \notag \\
\end{eqnarray}

\noindent Eigenvalues in these ten sets have the following respective multiplicities: 
\begin{enumerate}
\item[1)] $1$;
\item[2)] $2$~\text{if}~$\left(  j_{1} \in \left\{ 2,3 \right\} \right)$~\text{and}\\
$0$~\text{otherwise};
\item[3)] $1$~\text{if}~$\left( j_{1}=3 \right)$~\text{and}\\
$0$~\text{otherwise};
\item[4)] $2^{n}$~\text{if}~$ \left( d_{n} \neq 0~\text{and}~(m-1)j_{n}+1 < w_{n} < mj_{n}-1 \right)$~\text{and}\\
$0$~\text{otherwise};
\item[5)] $2^{n-1}(j_{n}-2) I_{n-1} -2^{n} (1+ \lceil w_{n} \rceil -2m)$~\text{if}~$\left( (m-1)j_{n}+1 \leq w_{n} \leq mj_{n}-1 \right)$, \\
$2^{n-1}(j_{n}-2)I_{n-1}-m2^{n}(j_{n}-2)$~\text{if}~$ \left( mj_{n}-1 \leq w_{n} \leq mj_{n}+1 \right)$, ~\text{and} \\
$0$~\text{otherwise};
\item[6)] $2^{n-1}$~\text{if}~$\left( d_{n} \neq 0~\text{and}~mj_{n}-1 < w_{n} < mj_{n}+1 \right)$ ~\text{and}\\
$0$~\text{otherwise};
\item[7)] $2^{n-1}$~\text{if}~$\left( mj_{n}-1<w_{n} \leq mj_{n} \right)$ ~\text{and}\\
$0$~\text{otherwise};
\item[8)] $2^{n-1}$~\text{if}~$\left( mj_{n}-1< w_{n} < mj_{n} \right)$~\text{and}\\
$0$~\text{otherwise};
\item[9)] $2^{n-1}(I_{n-1}-1)-(m-1)2^{n}$~\text{if}~$\left( (m-1)j_{n}+1\leq w_{n} \leq mj_{n}-1 \right)$,\\
$2^{n-1}(I_{n-1}-1)-m2^{n}$~\text{if}~ $\left( mj_{n}-1<w_{n} \leq mj_{n}+1 \right)$, ~\text{and}\\
$0$~\text{otherwise};
\item[10)] $2^{n-2}(I_{n-1}-1)-(m-1)2^{n-1}$~\text{if}~$\left( (m-1)j_{n}+1\leq w_{n} \leq mj_{n}-1 \right)$, \\
$2^{n-2}(I_{n-1}-1)-m2^{n-1}$~\text{if}~$\left( mj_{n}-1<w_{n} \leq mj_{n}+1 \right)$, ~\text{and} \\
$0$~\text{otherwise}.
\end{enumerate}
\end{theorem}

Theorem \ref{thrm: SWspectrum} is a consequence of Theorem \ref{SpectrumOrtho} and the following lemmas. 

\begin{lemma}
\label{SWSpectrumF_{0}}
For $n=0$, $\sigma(A_{0}|_{D_{0}^{\prime}})= \bigcup_{k=1}^\infty \left\{ 4 k^{2} \pi^{2} \right\}$ with multiplicity one for all $k$. 
\end{lemma}

\begin{proof}
We look for eigenfunctions of $A_{0}$ on $F_{0}$ in $\Dom(A_{0})$. The only functions that satisfy these restrictions are translations of $\left\{\sin(2 k \pi x ) \right\}~\forall k \in \mathbb{N}$ supported inside the square well. We immediately obtain the eigenvalues, each with multiplicity one. 
\end{proof}

The set of eigenvalues in Lemma \ref{SWSpectrumF_{0}} comprises the first union in Theorem \ref{thrm: SWspectrum}. Now, we tackle the most complicated set, namely 
\begin{equation*}
\bigcup_{n=2}^{\infty} \bigcup_{k=1}^\infty  \left\{ \frac{ k^{2} \pi^{2}}{(d_{n}+\frac{1}{I_{n}})^{2}}, 
 \right\}
 \end{equation*}
 with the understanding that the complete list of eigenvalues comes from an application of similar arguments in conjunction with Theorem \ref{SpectrumOrtho}. 
 
\begin{lemma}
\label{SplitCrosses}
The set of eigenvalues $\bigcup_{n=2}^{\infty} \bigcup_{k=1}^\infty  \left\{ \frac{ k^{2} \pi^{2}}{(d_{n}+\frac{1}{I_{n}})^{2}} \right\}$ has multiplicity  
\begin{equation*}
2^{n-1}~\text{if}~\left( mj_{n}-1< w_{n} < mj_{n} \right)~\text{and}~0~\text{otherwise}~\forall~n\geq 2.
\end{equation*}
\end{lemma}
\begin{proof}
In keeping with \cite{RS09} and \cite{ST08}, we can decompose each quantum graph approximation $F_{n}$ into V's, loops, and crosses, and then deduce from the orthogonality conditions required by Theorem \ref{SpectrumOrtho} that each shape found within the square well contributes once to the overall multiplicity. The eigenvalues in the lemma come from crosses that straddle the boundary and whose centers lie in the square well. In this case, the eigenfunctions in $\Dom(A_{n})$ must take opposite values along the two X's that comprise the cross so that the eigenfunction is determined by the value it takes on the upper X. Therefore, one solution is to construct eigenfunctions that assume the same value on the upper and lower parts of the X and which vanish outside the square well and vanish on the corners inside the square well, namely 
\begin{equation*}
\sin\left( \frac{ k \pi } {d_{n} + \frac{1}{I_{n}}} (x-x_{0}) \right)~\text{for $x$ on the cross}~\text{and}~0~\text{otherwise}.
\end{equation*}
From this, we read off the associated eigenvalues 
\begin{equation*}
\left\{ \frac{ k^{2} \pi^{2} }{\left( d_{n}+\frac{1}{I_{n}} \right)^{2}} \right\}~\forall~k\in \mathbb{N}.
\end{equation*}
Lastly, it is shown in Lemma \ref{CrossLemma} that the number of such split crosses in $F_{n}$ is 
\begin{equation*}
2^{n-1}~\text{if}~\left( mj_{n}-1< w_{n} < mj_{n} \right)~\text{and}~0~\text{otherwise}. 
\end{equation*}
Since each split cross contributes one eigenvalue, we have the claimed multiplicity. 
\end{proof}

The remaining portion of this section gives the counts and placements for the shapes which comprise $F_{n}$. Combining these results with arguments similar to one in Lemma \ref{SplitCrosses} gives the eigenvalues and multiplicities in Theorem \ref{thrm: SWspectrum}.  We can derive from Lemmas 3.1, 3.2, and 3.3 in \cite{BDMS10} that V's occupy a total of two columns, loops occupy a total of $I_{n-1}(j_{n}-2)$ columns, and crosses occupy a total of $2(I_{n-1}-1)$ columns.  Thus, since there are $I_{n}$ columns in the $F_{n}$ quantum graph approximation, 
\begin{equation}
I_{n}=2+I_{n-1}(j_{n}-2)+2(I_{n-1}-1).\label{I_n short}
\end{equation}
Moreover, proposition 3.1 in \cite{BDMS10} implies that a V occupies the first column in $F_{n}$, loops occupy the next $j_{n}-2$ columns, and a cross occupies the next two columns.  Loops and crosses continue to alternate across $F_{n}$---loops arising in clusters of $j_{n}-2$ and crosses covering two columns each.  The last $j_{n}-1$ columns are occupied by loops and a V, respectively.  Thus, we can expand equation \ref{I_n short} into

\begin{equation}
I_{n}=1+(j_{n}-2)+2+(j_{n}-2)+2+\cdots+2+(j_{n}-2)+1.
\end{equation}
\vspace{1mm}

\begin{proposition}
Let m be an integer.
\begin{enumerate}
  \item[(a)] The column boundary of the V's is denoted by $[0, 1]$ and $[I_{n}-1, I_{n}]$.\\
  \item[(b)] The column boundary of the $m^{th}$ set of loops in any row is denoted by $[(m-1)j_{n}+1$,\hspace{1.5mm} $mj_{n}-1]$, where $1 \leq m \leq I_{n-1}$.\\
  \item[(c)] The column boundary of the $m^{th}$ cross in any row is denoted by $[mj_{n}-1$,\hspace{1.5mm} $mj_{n}+1]$, where $1 \leq m \leq I_{n-1}-1$
\end{enumerate}
\label{cross and loop boundary}
\end{proposition}

\begin{proof}
\begin{enumerate}
  \item[(a)]  Since V's are found at the edges of a Laakso space, they occupy the first and last columns.\\
  \item[(b)]  We note that the left-most set of loops in $F_{n}$ has a column boundary of $[1$, $j_{n}-1]$.  Suppose that the $m^{th}$ set of loops occupies the column $[(m-1)j_{n}+1$, $mj_{n}-1]$.  We are going to show that the $(m-1)^{th}$ set of loops must occupy the column $[(m-2)j_{n}+1$, $(m-1)j_{n}-1]$. First, we subtract two from $(m-1)j_{n}+1$ to get the upper bound of the $(m-1)^{th}$ loop.  Subtracting an additional $j_{n}-2$ gives us the lower bound.  Thus, by induction, we have shown that the $m^{th}$ loop occupies the column $[(m-1)j_{n}+1$, $mj_{n}-1]$.\\
   \item[(c)]  We note that the left most cross in $F_{n}$ has a column boundary of $[j_{n}-1$, $j_{n}+1]$.  Let's assume that the $m^{th}$ cross occupies the column $[mj_{n}-1$, $mj_{n}+1]$.  We are going to show that the $(m-1)^{th}$ cross must occupy the column $[(m-1)j_{n}-1$, $(m-1)j_{n}+1]$.  First, we subtract $j_{n}-2$ from $mj_{n}-1$ to get the upper bound of the $(m-1)^{th}$ cross.  Subtracting an additional two columns gives us the lower bound.  Thus, by induction, we have shown that the $m^{th}$ cross occupies the column $[mj_{n}-1$, $mj_{n}+1]$.
\end{enumerate}
\end{proof}

\begin{lemma}
\label{CrossLemma}
Let $w_{n},m > 0$.
\begin{enumerate}

  \item[(a)]  When $(m-1)j_{n} < w_{n} \leq mj_{n}-1$, then there are

  \begin{equation}
  2^{n-2}(I_{n-1}-1)-(m-1)2^{n-1}
  \label{crosses in SW1}
  \end{equation}\

  \noindent full crosses on the interior of the square well.

  \item[(b)]  When $mj_{n}-1 < w_{n} \leq mj_{n}$, there are

  \begin{equation}
  2^{n-2}(I_{n-1}-1)-(m-1)2^{n-1}
  \label{crosses in SW2}
  \end{equation}\

  \noindent full crosses on the interior of the square well and $2^{n-1}$ half crosses.

\end{enumerate}
\end{lemma}

\begin{proof}
\begin{enumerate}
  \item[(a)]  When $w_{n} \in ((m-1)j_{n}, mj_{n}-1]$, then the $(m-1)^{th}$ cross in any given row either straddles the wall of the square well or is outside the square well.  There are $2^{n-2}$ rows of crosses in $F_{n}$, so we multiply this number by $2(m-1)$, which gives us the total number of crosses that are not in the square well.  Finally, we subtract the resulting number from the total number of crosses in the graph, giving us the formula in \ref{crosses in SW1}.  Note that if $x=\frac{1}{4}$ intersects the $(m-1)^{th}$ cross, then more than half of the cross will remain on the exterior of the square well.  Thus, the square well will never contain a half cross.\\
  \item[(b)]  When $w_{n} \in (mj_{n}-1, mj_{n}]$, then $x=\frac{1}{4}$ intersects the left half of the $m^{th}$ cross in any given row. Although the entire cross is not in the square well, the right half- cross is.  So, there are $2^{n-2}(I_{n-1}-1)-m2^{n-1}$ intact crosses and $2^{n-1}$ half-crosses.
\end{enumerate}
\end{proof}

\begin{lemma}
Let $w_{n},m > 0$.
\begin{enumerate}

  \item[(a)]  When $(m-1)j_{n} < w_{n} \leq mj_{n}-1$, then there are

  \begin{equation}
  2^{n-1}(j_{n}-2)(I_{n-1})-2^{n}(1+\lceil w_{n} \rceil -2m)
  \label{loops in SW1}
  \end{equation}\

  \noindent loops on the interior of the square well.

  \item[(b)]  When $mj_{n}-1 < w_{n} \leq mj_{n}$, there are

  \begin{equation}
  2^{n-1}(j_{n}-2)(I_{n-1})-m2^{n}(j_{n}-2)
  \label{loops in SW2}
  \end{equation}\

  \noindent loops on the interior of the square well.

\end{enumerate}
\end{lemma}

\begin{proof}
\begin{enumerate}
  \item[(a)]  Consider when $(m-1)j_{n} < w_{n} \leq mj_{n}-1$.  In any given row, we note that $mj_{n}-1-\lceil w_{n} \rceil$ gives us the number of loops in the $m^{th}$ set of loops that falls to the right of $x=\frac{1}{4}$.  We subtract this number from the total number of loops in that cluster, giving us $j_{n}-2-(mj_{n}-1-\lceil w_{n} \rceil)$.   There are $(j_{n}-2)(m-1)$ remaining loops to the left of $x=\frac{1}{4}$, so we add these two terms together and multiply the sum by $2^{n-1}$ giving us $2^{n-1}(1+\lceil w \rceil-2m)$.  Since we have only accounted for the total number of loops to the left or on $x=\frac{1}{4}$, we must multiply the previous value by two, giving us $2^{n}(1+\lceil w_{n} \rceil-2m)$.  Finally, we subtract $2^{n}(1+\lceil w \rceil-2m)$ from the total number of loops in the graph, giving us \ref{loops in SW1}.\\
  \item[(b)] When $mj_{n}-1 < w_{n} \leq mj_{n}$, then the number of loops on the exterior of the infinite square well are the same as the number of loops on the exterior when $w_{n}=mj_{n}-1$.  Part a implies that when $w_{n}=mj_{n}-1$, then the number of loops on the interior of the square well is $2^{n-1}(j_{n}-2)(I_{n-1})-m2^{n}(j_{n}-2)$.

\end{enumerate}
\end{proof}

\subsection{Coulomb and Parabolic Potentials} \label{subsect: CP}
The next two potentials we wish to discuss are the Coloumb Potential and the Parabolic Potential.  The Hamiltonian with Coloumb potential is $H_{C}=\Delta$ + $V(x)$, where $V(x)$ is given by

\begin{equation}
V(x) = \frac{-1}{(x-\frac{1}{2})^2} + \frac{1}{4}.
\end{equation}

\noindent In Section \ref{sect:NM}, we provide the bottom of $H_{C}$'s spectrum (Table \ref{table: CP}) and a graph of the eigenfunction corresponding to the smallest eigenvalue (Figure \ref{fig: CP}). As expected, the eigenfunction is zero at $x = \frac{1}{2}$. 

The Hamiltonian with parabolic potential is $H_{P}=\Delta$ + $V(x)$, with $V(x)$ defined as

\begin{equation}
V(x) = \frac{1}{x(1-x)}.
\end{equation}

\noindent In Section \ref{sect:NM}, we provide the bottom of $H_{P}$'s spectrum (Table \ref{table: CP}) and a graph of the eigenfunction corresponding to the smallest eigenvalue (Figure \ref{fig: CP}). As expected,  the eigenfunction is zero at $x = 0$ and $x=1$.

In general there is no method for calculating the closed-form solutions to a linear second-order differential equation. However, by assuming a Taylor expansion for such a solution in the presence of a locally linear potential one can see that the eigenvalues to depend not only on the constant term but also on the first order term in the potential.

\subsection{Numerical Methods} \label{sect:NM}
We modified the MatLab script used in \cite{RS09} and \cite{ST08} to calculate the eigenvalues of the Hamiltonian with certain potentials of classical interest. In all three cases, $V(x)$ was represented by a diagonal matrix with large finite cut-offs to approximate infinity. 

\begin{table}[t]
\small
\begin{center}
\begin{tabular} {| c |  c  |  c  | c  |  c |}
\hline
\multicolumn{5}{|c|}{$H_{SW}$ Spectrum} \\
\hline
\multicolumn{2}{|c|}{Bound: $10^{15}$} & \multicolumn{2}{|c|}{Bound: $10^{15}$} & Expected\\
\hline
\multicolumn{2}{|c|}{$n = 6$} & \multicolumn{2}{|c|}{$n = 8$}  & \\
\hline
$\lambda$ & m & $\lambda$ & m & \\
\hline

38.1& 1 & 39 &1 & $(2\pi)^2$ = 39.48\\
\hline
88.8 & 1 & 89 & 1 & $(3\pi)^2$ = 88.83\\
\hline
152.2 & 3 &  157 & 3& $(4\pi)^2$ = 157.91 \\
\hline
342.3 & 1 & & &  \\
\hline
355.1 & 9 & 353 & 10& $(6\pi)^2$ = 355.31\\
\hline
608.2 & 3 & 628 & 3& $(8\pi)^2$ = 631.65 \\
\hline
798.31 & 1 & 799 & 1& $(9\pi)^2$ = 799.44\\
\hline
 949.8 & 1 & 981 & 1& $(10\pi)^2$ = 986.96\\
\hline
1272.2 & 4 & && \\
\hline
1366.7 & 3&1395.4& 4 &\\
\hline
1417.6 & 21&1412&24& $(12\pi)^2$ = 1421.22 \\
\hline
1858.8 & 1 &1922&1& $(14\pi)^2$ = 1934.44\\
\hline
2211.9 & 1& 2220&1&$(15\pi)^2$ = 2220.7 \\
\hline
2425.0 & 3 &2511&3& $(16\pi)^2$ = 2526.62\\
\hline
3065.6 & 1&3178&1&\\
\hline
3179.5 & 29&3197&29&$(18\pi)^2 = 3197.8$\\
\hline
3779.8 &	3 &3923&3& $(20\pi)^2$ = 3947.84\\
\hline
4318.8 &1&4253&1 & $(21\pi)^2 = 4352.50$\\
\hline
4567 & 1 &4746&1& $(22\pi)^2$ = 4776.89  \\
\hline
5055.9 & 4 & 5580&4&$(24\pi)^2 = 5684.89$\\
\hline

\end{tabular}

\end{center}
\caption{The first 20 eigenvalues computed for $H_{SW}$ and $j_{n} = [2,3,2,3...]$.}
\label{table: SW}
\end{table}

In Table \ref{table: SW}, we compare the MatLab calculations for $H_{SW}$ with eigenvalues found in Theorem \ref{thrm: SWspectrum}.  These columns match for most values, but MatLab gives additional quantities as well. However, as the finite cut-off grows, these extraneous eigenvalues are lost, allowing for the MatLab calculations to coincide with the predicted spectrum. 

\begin{table}[t]
\small
\begin{center}
\begin{tabular} {| c |  c  |  c  | c | }

\hline
\multicolumn{2}{|c|}{$H_{C}$ Spectrum} & \multicolumn{2}{|c|}{$H_{P}$ Spectrum}\\
\hline
\multicolumn{2}{|c|}{Bound: $-10^{15}$}  & \multicolumn{2}{|c|}{Bound: $10^{15}$}\\
\hline
$\lambda$ & m & $\lambda$ & m \\
\hline

-1391.7&	16 &14.7& 1\\
\hline
-.6 &	4  & 45.7&3\\
\hline
64.9& 4 &92.9&1\\
\hline
87.8 & 4&95.8&1\\
\hline
 227.5  & 4&165.4& 3\\
\hline
318.8 & 4&254.7&1\\
\hline
328.4 & 4&359.1&2\\
\hline
342.3 & 4&359.2&3\\
\hline
 349.8 & 4&359.6&4\\
\hline
354.6& 8&360.5&4\\
\hline
478.9 & 4&362.5&4\\
\hline
 796.4 & 4&363.5&3\\
\hline
816.5 & 4&370.7&4\\
\hline
1238.4& 4&491.9&1\\
\hline
1354.6 & 4& 639.9&3\\
\hline
1376.5 & 4&802.4&1 \\
\hline
1398.2 &	4 &807.5&1 \\
\hline
1404.0 & 4 &994.6&3 \\
\hline
1409.8 & 4 &1201.3&1\\
\hline
1412.1 & 4 &1421.7&6 \\
\hline

\end{tabular}

\end{center} 
\caption{The first 20 eigenvalues for $H_{C}$ and $H_{P}$ and $j_{n} =  [2,3,2,3,2,3]$}
\label{table: CP}
\end{table} 

In Table \ref{table: CP}, we give the 20 eigenvalues closest to zero for $H_{C}$ and $H_{P}$. As expected, we have some negative values in the spectrum of $H_{C}$. For $H_{C}$, we approximated $V(x)$ at $x = \frac{1}{2}$ to be $-10^{15}$, and for $H_{P}$ we approximated $V(x)$ at $x = 0$ and $x = 1$ to be $10^{15}$.
    
\begin{figure}[htbp]
\begin{center}
 \includegraphics[scale=.42] {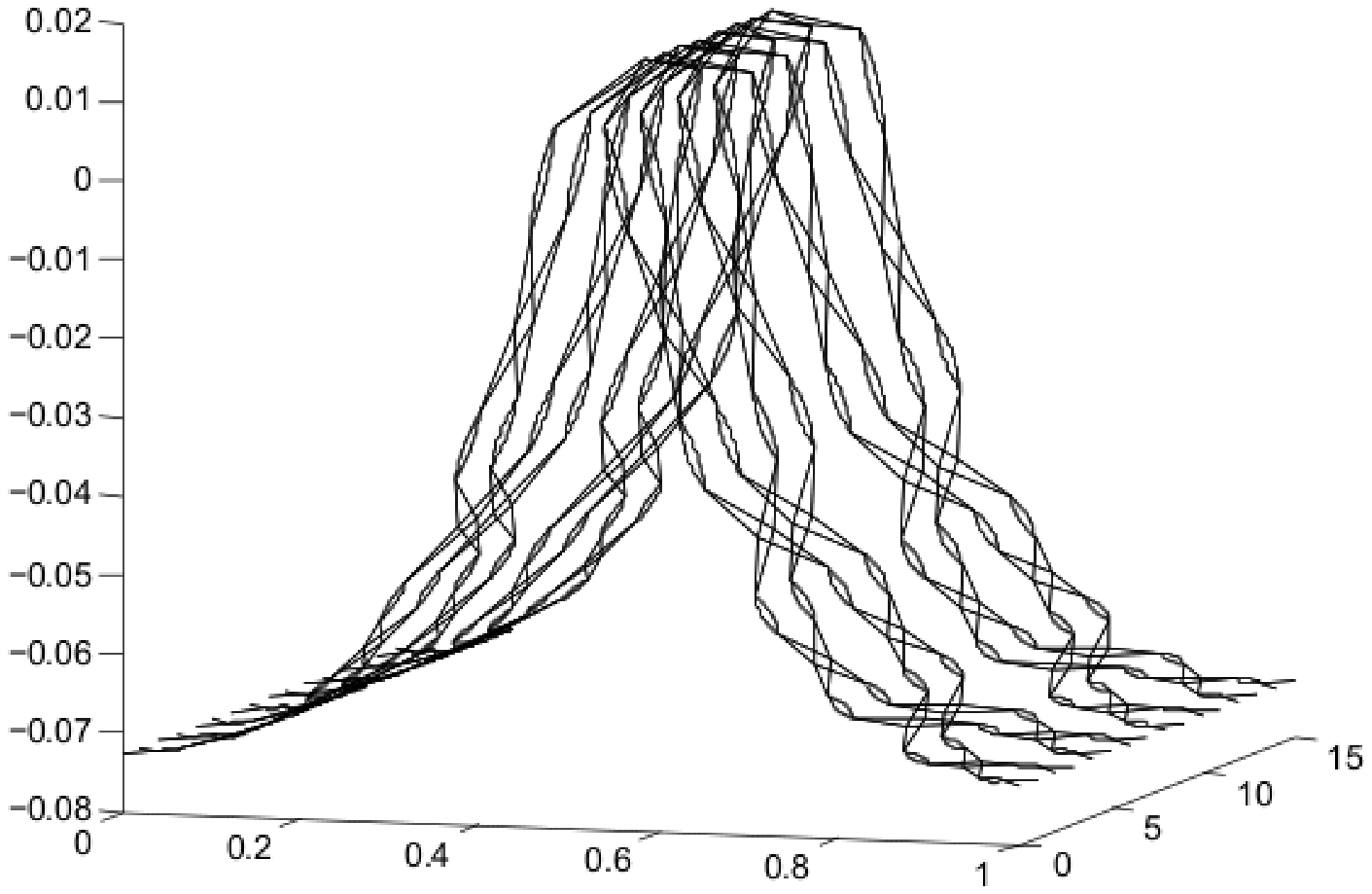} \includegraphics[scale=.42]{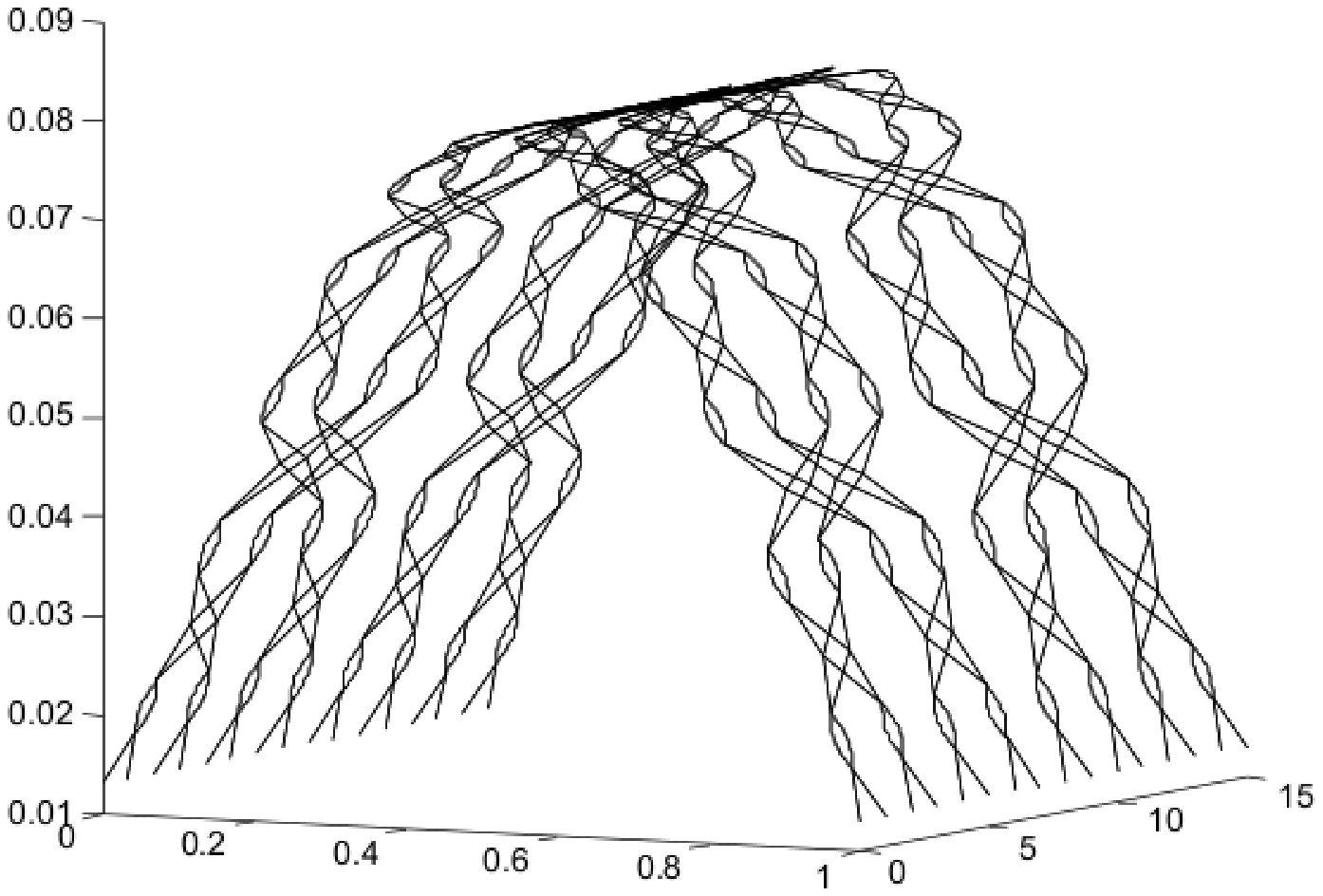} 
 $H_{C} \hspace{6.5cm}$ $H_{P}$ 
 \end{center}
\caption[]{$H_{C}$ and $H_{P}$ with $j_{n}$ = [2,3,2,3]: Eigenfunction 1}
\label{fig: CP}
\end{figure}

In Figure \ref{fig: CP}, we see the first eigenfunction of $H_{C}$ and  $H_{P}$, respectively. Notice the position of zero for these eigenfunctions. For $H_{C}$, the eigenfunction is zero at $x = \frac{1}{2}$, and for $H_{P}$, it is zero at $x = 0$ and $x = 1$. This is exactly where the potentials are infinite.

\begin{figure}[htbp]
\begin{center}
 \includegraphics[scale=.5] {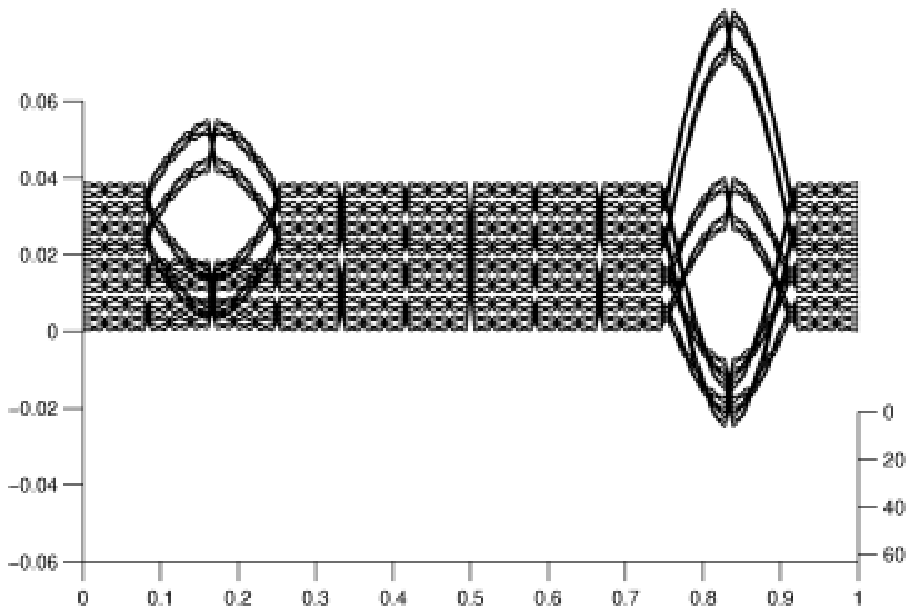} \includegraphics[scale=.5]{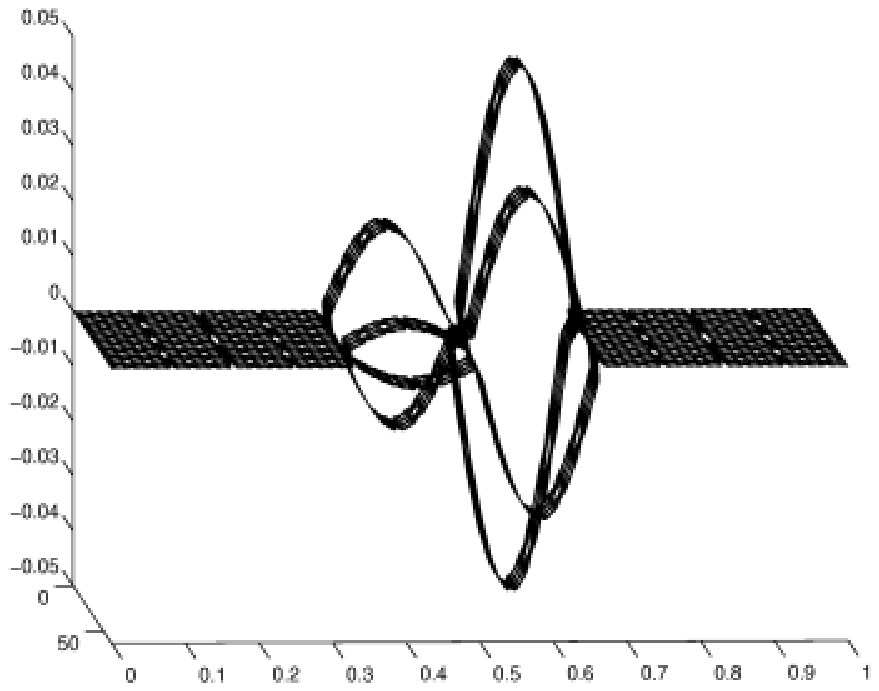} 
 $H_{C} \hspace{6.5cm}$ $H_{P}$ 
 \end{center}
\caption{$H_{C}$ and $H_{P}$ with $j_{n}$ = [2,3,2,3]: Eigenfunction 30 and 15 respectively}
\end{figure}

\section{Electromagnetic Fields and Conducting Plates in Laakso Spaces}\label{sect: Casimir}
One of the simplest and most frequently mentioned cases in Quantum Mechanics is the harmonic oscillator, which we conceptualize as a block attached to a massless spring of some fixed equilibrium length \cite{Griffiths2005}. To describe the mechanics of such a system, we introduce a conserved Hamiltonian
\begin{equation}
 H=E_{Kinetic}~+~E_{Potential}= \frac{1}{2} \left( \frac{p^{2}}{m}+m \omega^{2} q^{2} \right),
\end{equation}
where 
\begin{align*}
m= &\textit{ mass ~of ~block ~attached ~to~ spring,} \\
q=& \textit{ displacement~ of ~block ~from ~equilibrium,} \\
p=& \textit{ momentum ~of~ block,} \\
 \omega=& \textit{constant involving spring constant and mass}.
 \end{align*}

\vspace{2mm}

\noindent We can also consider the Hamiltonian and its components as operators:
\begin{align*}
p[f]= &-i\hbar \frac{d}{dx}f \\ 
q[f]=&qf \\
 H[f]=&\frac{1}{2} \left( \frac{p^{2}}{m}+m\omega^{2} q^{2} \right)[f]=\frac{1}{2} \left(\frac{-\hbar^{2}}{m} \frac{d^{2}}{dx^{2}}f+ m \omega^{2}q^{2}f \right).
\end{align*}

The importance of this configuration becomes apparent in \cite{Bog1980} where the quantized Hamiltonian of the electromagnetic field is shown to be equivalent to the Hamiltonian of a set of independent harmonic oscillators,
\begin{equation}
H_{EM}=\frac{1}{2} \sum_{k,s} \left( \frac{ p_{k,s}^{2}(t) } {m} +m \omega_{k}^{2} q_{k,s}^{2}(t) \right),
\end{equation}
where $p_{k,s}(t)$ and $ q_{k,s} (t)$ are time-dependent functions of the electromagnetic field, $\left\{ \omega_{k}  \right\} = \left\{ \sqrt{ \lambda}~ | ~\lambda \in \sigma{(\Delta)} \right\}$, \emph{k} is the wave vector of the radiation, and \emph{s} is the degree of freedom resulting from the different possible polarizations. Given this equivalence, one should not be surprised to find that, just as an isolated oscillator assumes a discrete set of allowable energies given by 
\begin{equation*}
\left\{  \left.\hbar  \omega \left( n+\frac{1}{2} \right) \right | n \in \mathbb{N} \cup{\left\{ 0 \right\}} \right\},
\end{equation*}
the energies for the quantized electromagnetic field are
 \begin{equation*}
 \left\{ \left. \sum_{j} \hbar \omega_{j} \left(n_{j}+\frac{1}{2} \right) \right | n_{j} \in \mathbb{N}\cup{\left\{0 \right\}}~\forall j \right\}.
 \end{equation*}
Similarly, the respective zero-point energies are
\begin{eqnarray}
E_{0}=&\left.  \hbar  \omega (n+\frac{1}{2}) \right |_{n=0} &=\frac{\hbar \omega}{2} \\
\label{Zero-Point} E_{0}=&\left. \sum_{j} \hbar \omega_{j} (n_{j}+\frac{1}{2}) \right |_{(n_{1},n_{2}, \dots)=0}&=\frac{ \hbar} {2} \sum_{n} \omega_{n},
\end{eqnarray} 
where the sum in Equation \ref{Zero-Point} is over $\left\{ \omega_{n}  \right\} = \left\{ \sqrt{ \lambda}~ | ~\lambda \in \sigma{(\Delta)} \right\}$ \cite{Berkolaiko2009}. An immediate consequence of this equation is that free space has a non-zero minimum energy density. Another perhaps less obvious result is derived in a 1948 paper by H.B.G. Casimir, which claimes that under appropriate boundary conditions, two uncharged conducting plates should experience a mutual force of attraction which varies with the inverse fourth power of their separation distance  \cite{Casimir1948}. Specifically, 
\begin{equation}
\left. \right | F_{C} \left. \right |=\frac{ \pi^{2} c \hbar}{240 a^{4}}.
\end{equation}
In the years following Casimir's original paper, mathematicians have made numerous attempts to verify an attractive force between two uncharged conducting plates. In 1957, M. J. Sparnaay found Casimir's expected value to be within his experimental margin of error in \cite{Sparnaay1957}, while in 2002 a research team at the University of Padua experimentally confirmed this formula in \cite{Bressi2002}. Here, we derive a general formula for the Casimir force in a fractal setting.
  
 Fix the following configuration: two uncharged conducting plates are symmetrically placed in a Laakso Space and attached to all nodes that intersect the conductors in quantum graph approximations. To ease the calculation, we only consider conducting plates which attach to nodes in $F_{1}$ and require that eigenfunctions of the Laplacian satisfy Dirichlet conditions at conducting nodes and Kirchoff conditions at non-conducting nodes. At non-conducting nodes of degree one, this last requirement is equivalent to the Neumann condition that requires nodal derivatives to vanish. Moreover, these two restrictions preserve the self-adjointness of the Laplacian operator and incorporate appropriate boundary conditions at the plate by analogy to the laws of classical electrostatics. Most importantly, the two conducting plates within the Laakso Space are allowed to move as long as the symmetry between the two plates is preserved and the Laakso space is appropriately distorted.  If the plates move towards one another, intervals in the region between the plates of the quantum graph approximation are compressed, and intervals in the two exterior regions are stretched. Conversely, if the plates move apart, intervals in the interior region are stretched, and intervals in the exterior regions are compressed. 
 
\begin{definition}
\begin{enumerate}
\setlength{\itemsep}{2mm}
\item The $j=N$ Laakso space has $j_{n}=N \hspace{1mm}  ~\forall n$. 

\item $X_{0} =$ the distance between a conducting plate and the center of the Laakso space. 

\item $Z = $ the number of nodes lying between the two conducting plates in the $F_{1}$ quantum graph.

\end{enumerate}
\end{definition}

\vspace{2mm}

\begin{lemma}
Let two uncharged conducting plates be symmetrically attached to nodes in the $F_{1}$ quantum graph approximation so that Z nodes in $F_{1}$ lie between them. Then
\begin{enumerate}
 \item The number of cells that lie between the plates in $F_{n}$ is $1$ if $n=0$ and $\frac{ (Z+1)}{N}2^{n}I_{n} ~if~n \geq 1$. The number of exterior cells is $2$ if $n=0$, and $(1-\frac{Z+1}{N})2^{n}I_{n} ~if~n \geq 1$.\\
\item The number of interior loops in $F_{n}$ is $Z+1$ if $n=1$, and $\frac{Z+1}{N} 2^{n-1}I_{n-1}(j_{n}-2)$ if $n \geq 2$. The number of exterior loops in $F_{n}$ is $N-Z-3$ if $n=1$, and $(1-\frac{Z+1}{N})(2^{n-1}I_{n-1})(j_{n}-2)$ if $n \geq 2$.\\
 \item All V's are located in the exterior region, and there are $2^{n}$ of them in $F_{n}$ for $n \geq 1.$ \\
 \item The number of interior crosses in $F_{n}$ is $\frac{Z+1}{N}2^{n-2}I_{n-1} ~\forall n \geq 2$ , while the number of exterior crosses is $2^{n-2}[(1-\frac{Z+1}{N})I_{n-1}-1] ~ \forall n  \geq 2$.
\end{enumerate}
\end{lemma}

 \begin{proof}
\begin{enumerate}
\item \emph{Interior Region}. By the self-similarity of the $j=N$ Laakso space, we know that $\forall n,k \in \mathbb{N}\cup \left\{0 \right\}$ and $\forall k: 0 \leq k<N$, the number of cells with $x$-coordinate in $[k/N,(k+1)/N]$ are equal for all $k$. We have already shown that the $F_{n}$ graph approximation contains $2^{n}I_{n}$ cells and see that $Z+1$ almost disjoint intervals of the form $[k/N,(k+1)/N]$  comprise the region between the conducting plates. Therefore, since each region has $2^{n}I_{n}/N$ cells in $F_{n}$ and $Z+1$ regions lie between the plates, $\frac{Z+1}{N} 2^{n} I_{n}$ cells lie between the conductors in $F_{n}$.\\
\indent \emph{Exterior Region}.  Subtracting the number of interior cells in $F_{n}$ from from the total yields the number of exterior cells: $(1-\frac{Z+1}{N})2^{n}I_{n}$. \\
\item \emph{Interior Region}.  At the first level of construction, we have $Z-2$ loops. At the $n-1$ stage of construction, we have $\frac{(Z+1)}{N}2^{n-1}I_{(n-1)}$ cells in the interior region. Since every cell at the $n-1$ stage will yield exactly $j_{n}-2$ loops in the $n$ level, we compute that $\frac{Z+1}{N} 2^{n-1}I_{n-1}(j_{n}-2)$ loops lie in the interior region of $F_{n}$ for $n \geq 2$.\\
\indent \emph{Exterior Region}.  Since $F_{1}$ has a total of $N-2$ loops and the interior region has only $Z+1$ of those loops, we are left with $N-Z-3$ loops in the exterior region for $n=1$. Since $(1-\frac{Z+1}{N})(2^{n-1}I_{n-1})$ cells inhabit the exterior region of $F_{n}$, we have by the same argument as before that $(1-\frac{Z+1}{N})(2^{n-1}I_{n-1})(j_{n}-2)$ loops are to be found in $F_{n}$ whenever $n \geq 2$. \\
\item This follows immediately from the Barlow-Evans construction of Laakso Spaces. \\
\item \emph{Interior Region}.  Since all V's are located in the exterior region, the only possible structures in the interior region are loops and crosses. Moreover, we know the total number of loops and the total number of cells in the interior region for any value for n, so letting $X_{n}$ =(number of interior $F_{N}$ cells in crosses) yields
\begin{equation*}
\frac{Z+1}{N}2^{n}I_{n-1}(j_{n}-2)~+~X_{n}=\frac{Z+1}{N}2^{n}I_{n}
\end{equation*}
\begin{equation*}
\implies X_{n}=\frac{Z+1}{N}2^{n+1}I_{n-1}.
\end{equation*} 
Since every cross is comprised of exactly 8 cells, we simply divide by 8 to calculate the total number of crosses in the interior region:
\begin{equation*}
\frac{Z+1}{N}2^{n-2}I_{n-1} ~\forall n \geq 2.
\end{equation*}
The total counts two disjoint half-crosses as one whole cross.\\ 
\indent \emph{Exterior Region}.  Every exterior cell in $F_{n}$ forms either part of a cross, loop, or V. Therefore, to find $X_{n}$=(number of exterior $F_{n}$ cells in crosses), we simply solve

\begin{equation*} 
2^{n+1}+\left\{ 1-\frac{Z+1}{N} \right\} 2^{n}I_{n-1}(j_{n}-2) +X_{n}= \left\{ 1-\frac{Z+1}{N} \right\} 2^{n}I_n 
\end{equation*}
\begin{equation*}
\implies  X_{n}= 2^{n+1} \left\{ \left\{ 1-\frac{Z+1}{N} \right\} I_{n-1}-1 \right\}. 
\end{equation*}
We divide the result by 8 to obtain the total number of exterior crosses:
\begin{equation}
2^{n-2}\left\{ \left\{ 1-\frac{Z+1}{N} \right\} I_{n-1}-1 \right\} ~ \forall n \geq 2. 
\end{equation}
\end{enumerate}
\end{proof} 
When counting the total number of crosses in the interior and exterior regions, we also treat two disjoint half-crosses as one whole cross. 

\begin{proposition}
The total number of crosses centered along the two conducting plates in $F_{n}$ is $2^{n-1} ~\forall n \geq 2$. 
\end{proposition}

\begin{proof}
In $F_{1}$ we have two conducting nodes. Since each of these two nodes is degree four, we will have at the second level of construction two crosses centered over the conductors. Each of these new crosses will contribute two additional conducting nodes, so that four degree four nodes lie across the conductors in $F_{2}$.  By induction, we arrive at a general formula for the total number of crosses centered along the two conducting plates in $F_{n}$: $2^{n-1} ~\forall n \geq 2$.
\end{proof}
\begin{corollary} 
There are $2^{n-1}$ half-crosses lying along the conducting plates in each of the interior and exterior regions.
\end{corollary}
\begin{proof}
By symmetry, we may deduce that the  number of half-crosses along the conductors in the exterior region is always equal to the number of half-crosses in the interior region.
\end{proof}

\begin{lemma}
If conducting plates are attached to two symmetric nodes of a $j=N$ Laakso space so that $Z$ nodes lie between the plates in $F_{1}$, and the plates are moved to a distance $X_{0}$ from the center, allowing interior and exterior regions of the Laakso space to expand and contract, then $~\forall n \geq 1$, the interior and exterior cell lengths in $F_{n}$ will be functions of $X_{0}$ given by $\frac{2NX_{0}}{Z+1}I_{n}^{-1}$ and $\frac{1-2X_{0}}{1-\frac{Z+1}{N}}I_{n}^{-1}$, respectively. 
\end{lemma}

\begin{proof}
Each cell in the unperturbed $F_{1}$ quantum graph approximation has length 1/N. If we move the conducting plates to a distance $X_{0}$ from the center, every $F_{1}$ cell in the interior region will have length $\frac{2X_{0}}{Z+1}$. Therefore, the scaling factor is $\frac{2X_{0}N}{Z+1}$. Moreover, every $F_{1}$ cell in the exterior region will have length $\frac{1-2X_{0}}{N-(Z+1)}$ so that the scaling factor is $\frac{1-2X_{0}}{1-\frac{Z+1}{N}}$. Lastly, it is clear from the self-similarity of j=N Laakso spaces that interior and exterior cells from more detailed quantum graph approximations will also be scaled by the same ratios. 
\end{proof}
\begin{corollary}
The interior and exterior metric diameters in $F_{n}$ $~\forall n \geq 1$ will be functions of $X_{0}$ given by $\frac{Z+1}{2NX_{0}}I_{n}$ and $\frac{1-\frac{Z+1}{N}}{1-2X_{0}}I_{n}$, respectively. 
\end{corollary}

Cells in $F_{0}$ require separate treatment. 
\begin{proposition}
The interior region in $F_{0}$ has metric diameter $\frac{1}{2X_{0}}$, while each of the two exterior regions has metric diameter $\frac{2}{1-2X_{0}}$. 
\end{proposition}
\begin{proof}
The metric diameter of a cell is the multiplicative inverse of cell length. 
 \end{proof}
To demonstrate that the Laplacian operator on our modified space is self-adjoint, it suffices to prove the result on every quantum graph approximation and check for mutual compatibility. 

\begin{definition}
\label{Bdef}
The differential operator $B_{n}$ acts on $F_{n}$ by 
\begin{equation*}
B_{n}[f]=-\frac{d^{2}}{dx_{e}^{2}} f
\end{equation*}
along each edge e with $\Dom(B_{n})=\left\{ \right.$continuous $f \in H^{2}(e)~\forall e$ with Dirichlet vertex conditions at each conducting node and Kirchoff vertex conditions at each non-conducting node$\left.  \right\}$.
\end{definition}

\begin{theorem}
$B_{n}$ is a self-adjoint operator on $F_{n}$.
\end{theorem}
\begin{proof}
Let $v$ be a vertex of degree $d$ in $F_{n}$ and $f$ a function on $F_{n}$.  Define $F_{v}=(f_{1}(v),f_{2}(v),...,f_{d}(v))^{T}$ and $F^{\prime}_{v}=(f_{1}^{\prime}(v),f_{2}^{\prime}(v),...,f_{d}^{\prime}(v))^{T}$. By Theorem 3 in \cite{Kuchment2008}, to prove self-adjointness it suffices to find for each vertex in a quantum graph approximation matrices $C_{v}$ and $D_{v}$ such that $C_{v}F+D_{v}F^{\prime}=0$ and $(C_{v}D_{v})$ is maximal. Since this condition is local and both Dirichlet and Neumann vertex conditions produce self-adjoint operators on a quantum graph, the theorem follows.
\end{proof}

\begin{theorem}
\label{Spectrum}
Let $B_{n}$ be the operator in Definition \ref{Bdef}, and let $\Delta$ be the minimal self-adjoint extension of the sequence $\left\{ B_{n} \right\}$. Then
\begin{eqnarray*}
\sigma(\Delta)=&& \bigcup_{k=1}^\infty \left\{\left[ \frac{ k \pi}{2X_{0}} \right]^{2}\right\} \cup  \bigcup_{k=0}^\infty \left\{\left[ \frac{ (k+1/2) \pi}{(1-2X_{0})/2} \right]^{2}\right\}\\& \cup& \bigcup_{k=0}^\infty \left\{\left[(k+1/2) \pi \frac{N-(Z+1)}{1-2X_{0}} \right]^{2}\right\}\cup \bigcup_{k=1}^\infty \left\{\left[k \pi \frac{N-(Z+1)}{1-2X_{0}} \right]^{2}\right\}\\ &\cup& \bigcup_{k=1}^\infty  \left\{\left[ k \pi \frac{Z+1}{2X_{0}} \right]^{2}\right\} \cup \bigcup_{n=2}^{\infty}  \bigcup_{k=0}^\infty \left\{\left[ I_{n}(k+1/2) \pi \frac{1-\frac{Z+1}{N}}{1-2X_{0}} \right]^{2}\right\}  \\ &\cup& \bigcup_{n=2}^{\infty} \bigcup_{k=1}^\infty \left\{\left[I_{n} k \pi \frac{1-\frac{Z+1}{N}}{1-2X_{0}} \right]^{2}\right\}  \cup \bigcup_{n=2}^{\infty} \bigcup_{k=1}^\infty \left\{\left[ I_{n} k \pi \frac{1-\frac{Z+1}{N}}{2(1-2X_{0})} \right]^{2}\right\} \\ &\cup&  \bigcup_{n=2}^{\infty} \bigcup_{k=1}^\infty \left\{\left[I_{n}k \pi \frac{Z+1}{2NX_{0}} \right]^{2}\right\} \cup \bigcup_{n=2}^{\infty} \bigcup_{k=1}^\infty \left\{\left[ I_{n}k \pi \frac{Z+1}{4NX_{0}} \right]^{2}\right\}.
\end{eqnarray*} 

\noindent Eigenvalues in these ten sets have the following respective multiplicities: 

\begin{enumerate}
\item[1)] $1$;
\item[2)] $2$;
\item[3)] $2$;
\item[4)] $N-Z-3$;
\item[5)] $Z+1$;
\item[6)] $2^{n}$;
\item[7)] $(1-\frac{Z+1}{N})I_{n-1}2^{n-1}(N-2)+2^{n-1}(1-\frac{Z+1}{N})I_{n-1}$;
\item[8)] $2^{n-2}[(1-(\frac{Z+1}{N})I_{n-1}-1]-2^{n-2}$;
\item[9)] $\frac{Z+1}{N}I_{n-1}2^{n-1}(N-2)+2^{n-1}\frac{Z+1}{N}I_{n-1}+2^{n-1}$;
\item[10)] $2^{n-2}[\frac{Z+1}{N}I_{n-1}-1]$.
\end{enumerate}
\end{theorem}

Using Theorem \ref{SpectrumOrtho}, we break the proof of Theorem \ref{Spectrum} into separate lemmas.  
\begin{lemma}
\label{SpectrumF_{0}}
For $F_{0}, \sigma(B_{0}|_{D_{0}^{\prime}})= \bigcup_{k=0}^\infty \left\{ [\frac{ (k+1/2) \pi}{(1-2X_{0})/2}]^{2} \right\}\cup \bigcup_{k=1}^\infty \left\{ [\frac{ k \pi}{2X_{0}}]^{2} \right\}$ with multiplicities $2$ and $1$, respectively. 
\end{lemma}

\begin{proof}
We look for eigenfunctions of $B_{0}$ on $F_{0}$ in $\Dom(B_{0})$. The only functions that satisfy these restrictions are $\left\{ \cos(2(k+1/2)\pi x /(1-2X_{0})) \right\}$ on the two exterior regions and $\left\{\sin(k\pi x /(2X_{0})) \right\}$ in the interior region. We immediately obtain the eigenvalues.
\end{proof}

\begin{lemma}
\label{SpectrumF_{1}}
For $F_{1}$, 
\begin{eqnarray*}
\sigma(B_{1}|_{D_{1}^{\prime}})=&& \bigcup_{k=0}^\infty \left\{ [(k+1/2) \pi (N-(Z+1))/(1-2X_{0}) ]^{2} \right\} \\ &\cup& \bigcup_{k=1}^\infty \left\{ [k \pi (N-(Z+1))/(1-2X_{0})]^{2} \right\}\\ &\cup&  \bigcup_{k=1}^\infty \left\{ [k \pi (Z+1)/(2X_{0})]^{2} \right\} 
\end{eqnarray*}
Eigenvalues in these three sets have the following respective multiplicities:
\begin{enumerate}
\item[1)]  $2$;
\item[2)] $N-Z-3$; 
\item[3)] $Z+1$. 
\end{enumerate}
\end{lemma}

\begin{proof}
The only eigenfunctions on $F_{1}$ that are orthogonal to the pullback of functions in $D_{0}^{\prime}$ are given by functions that take opposite values in the two copies of $F_{0}$ that are glued together. For instance, the eigenfunctions on the V's must be those whose values on the top branch equal the negative values for the function on the bottom branch so that the function defined on the top branch determines the value on the lower branch. Looking for eigenfunctions with Neumann conditions and Dirichlet boundary conditions at opposite ends of an interval of length $(1-2X_{0})/(N-(Z+1))$ gives $\left\{\cos((k+1/2)\pi x (N-(Z+1))/(1-2X_{0})) \right\}$. Similarly, the only eigenfunctions for loops on $F_{1}$ that meet the orthogonality restrictions are those whose values on the top branch equal the negative values for the function on the bottom branch. Looking for eigenfunctions of $B_{1}$ that have Dirichlet boundary conditions at the ends of an interval of length $(1-2X_{0})/(N-(Z+1))$ and $(2X_{0})/ (Z+1)$ gives $\left\{\sin((k \pi x (N-(Z+1))/(1-2X_{0})) \right\}$ and $\left\{\sin(k \pi x (Z+1)/(2X_{0})) \right\}$, respectively.   Since $F_{1}$ has two V's, $N-Z-3$ loops in the exterior region, and $Z+1$ loops in the interior, we have the corresponding multiplicities.
\end{proof}  
 
\begin{lemma}
\label{SpectrumF_{n}}
 For $F_{n} : n \geq 2$, 
 \begin{eqnarray*}
 \sigma(B_{n}|_{D_{n}^{\prime}}) = && \bigcup_{k=0}^\infty \left\{ [I_{n}(k+1/2) \pi (1-(Z+1)/N)/(1-2X_{0}) ]^{2} \right\}\\  &\cup&  \bigcup_{k=1}^\infty \left\{ [I_{n} k \pi (1-(Z+1)/N)/(1-2X_{0})]^{2} \right\} \\ &\cup&  \bigcup_{k=1}^\infty \left\{ [I_{n} k \pi (1-(Z+1)/N)/(2(1-2X_{0}))]^{2} \right\}  \\ &\cup& \bigcup_{k=1}^\infty \left\{ [I_{n}k \pi (Z+1)/(2NX_{0})]^{2} \right\} \\ &\cup&  \bigcup_{k=1}^\infty \left\{ [I_{n}k \pi (Z+1)/(4NX_{0})]^{2} \right\}
 \end{eqnarray*}
Eigenvalues in these five sets have the following respective multiplicities:
\begin{enumerate}
\item[1)] $2^{n}$; 
\item[2)] $(1-\frac{Z+1}{N})2^{n-1} I_{n-1}(N-2)+2^{n-1}(1-\frac{Z+1}{N})I_{n-1}$; 
\item[3)] $2^{n-2}[(1-\frac{Z+1}{N})I_{n-1}-1]-2^{n-2}$; 
\item[4)]\ $\frac{Z+1}{N}2^{n-1}I_{n-1} (N-2)+\frac{Z+1}{N} 2^{n-1} I_{n-1}+2^{n-1}$; 
\item[5)] $\frac{Z+1}{N} 2^{n-2}I_{n-1}-2^{n-2}$. 
\end{enumerate}
\end{lemma}

\begin{proof}
\noindent
\begin{enumerate}
\item \emph{V's}.  In keeping with the procedure introduced in \cite{RS09} and \cite{ST08}, we look for eigenfunctions of $B_{n}$ on $F_{n}$ in $\Dom(B_{n})$ which are orthogonal to the pullback of functions in $\Dom(B_{i})$ $\forall < n$. Consider the V's in $F_{n}$. The  orthogonality condition requires---as before---that the value of the eigenfunction on the top branch of a V equal the negative of the value on the lower branch so that the function on one branch, which vanishes at the junction of the V, fully determines the eigenfunction in the spectrum. By the definition of $\Dom(B_{n})$, the function must have Neumann boundary conditions at the two degree one nodes and Dirichlet boundary conditions at the degree two node. Therefore, a basis for the eigenfunctions of $B_{n}$ on each V of $F_{n}$ in $\Dom(B_{n})$ and orthogonal to the pullback of functions in $\Dom(B_{i}) ~\forall i<n$ is given by $ \left\{\cos(I_{n} (k+1/2) \pi x (1-(Z+1)/N)/(1-2X_{0})) \right\}$ and yields the set of eigenvalues $\cup_{k=0}^\infty \left\{ [I_{n}(k+1/2) \pi \left(1-\frac{Z+1}{N} \right)/(1-2X_{0}) ]^{2} \right\}$. Clearly, the total multiplicity of each eigenvalue for one V is $2^{n}$ since $2^{n}$ V's live in the exterior region of $F_{n}$.\\  
\item \emph{Exterior Loops}. Eigenfunctions must take opposite values on the two branches of the loop, and we must therefore only consider eigenfunctions on an interval of length $[I_{n} (1-\frac{Z+1}{N})/(1-2X_{0})]^{-1}$ with Dirichlet boundary conditions at the endpoints. We see immediately that the set of functions is $\left\{\sin(I_{n} k \pi x (1-\frac{Z+1}{N})/ (1-2X_{0})) \right\}$ which in turn yields a set of eigenvalues $\left\{ [I_{n} k \pi (1-(Z+1)/N)/(1-2X_{0})]^{2} \right\}$ each with multiplicity one. Since each loop will add the same set of eigenvalues, we use a previous result to see that each eigenvalue in the set  $\left\{ [I_{n} k \pi \left( 1-\frac{Z+1}{N} \right)/(1-2X_{0})]^{2} \right\}$ will have total multiplicity $\frac{Z+1}{N} 2^{n-1}$ $I_{n-1}(N-2)$.\\ 
\item \emph{Exterior Crosses}. The crosses are more complicated because $2^{n-1}$ of them are split along the conducting plates and so receive slightly different boundary conditions. For each cross not centered along a conducting plate, we treat an intact cross as two overlapping X's joined together at the four corner nodes. By the orthogonality condition, any permissible eigenfunction takes opposite values on the two X's of the cross so that once a function is applied to one X, the values of that function on the other cross are fully determined. From this, it is clear that the function must vanish at the four corners of the X. We can view this shape as an upper and lower V, joined at the central vertex, which is the $F_{1}$ quantum graph approximation for a $j=2$ Laakso space. Here, 
\begin{equation*}
\sigma(B_{0}|_{D_{0}^{\prime}})= \left\{ \left( \frac{k \pi I_{n} (1-\frac{Z+1}{N})}{2(1-2X_{0})} \right) ^{2} \right\}
\end{equation*}
where $\left\{ [I_{n} (1-(Z+1)/N)/[2(1-(2X_{0})]]^{-1} \right\} $ is the metric length of the upper branch of the X. We also know that 
\begin{equation*}
\sigma(B_{1}|_{D_{1}^{\prime}})= \left\{  \left( \frac{ k \pi I_{n} (1-\frac{Z+1}{N})}{(1-2X_{0})} \right) ^{2} \right\}
\end{equation*}
for each of the left and right side V's. Since this argument works only for those crosses which do not lie along the conducting plates, we have total multiplicities of $2^{n-2}[(1-(Z+1)/N)I_{n-1}-1]-2^{n-2}$ for $\left\{ \left(\frac{k \pi I_{n} (1-(Z+1)/N)}{2(1-2X_{0})} \right)^{2} \right\}$ and $2^{n-1}[(1-(Z+1)/N)I_{n-1}-1]-2^{n-1}$ for  $\left\{ \left( \frac{k \pi I_{n} (1-(Z+1)/N)}{(1-2X_{0})}\right)^{2} \right\}$. \\ 
\item \emph{Exterior Half-Crosses}.  For each cross centered on the conducting plates, take the half-cross lying in the exterior region, on which all permissible eigenfunctions take opposite values on the two halves of the half-cross. Thus, the function on the upper half of the half-cross fully determines the eigenfunction. Since we are imposing Dirichlet boundary conditions along the conductors, it is clear that the set of eigenfunctions is given by $\left\{ [k \pi I_{n} (1-\frac{Z+1}{N})/(1-2X_{0})]^{2} \right\}$, with multiplicity two to account for the fact that two separate intervals of length $\left\{ [I_{n} (1-\frac{Z+1}{N})/(1-2X_{0})]^{-1} \right\}$ form every half of the half-cross. Since there are exactly $2^{n-1}$ such half-crosses in the exterior region of $F_{n}$, we calculate a total multiplicity of $2^{n}$ for each of the eigenvalues listed in $\left\{ \left[ k \pi I_{n} \frac{1-\frac{Z+1}{N}}{(1-2X_{0})} \right]^{2} \right\}$. \\
\item \emph{Interior Region}. This case is handled by similar means.  
\end{enumerate}
\end{proof}
The proof of Theorem \ref{Spectrum} now follows from the results of Lemma \ref{SpectrumF_{0}}, Lemma \ref{SpectrumF_{1}}, Lemma \ref{SpectrumF_{n}}, and Theorem \ref{SpectrumOrtho}. 

Before calculating the Casimir Force in Subsection \ref{sect:CasimirForce}, we provide the following proviso.
\subsection{Proviso}
\label{Proviso}
Analytic regularization is a method employed by modern physics and, in particular, Quantum Field Theory to grapple with divergent sums which often arise in calculations of vacuum energy. More specifically, a divergent sum is interpreted to be the analytic continuation of a function that converges only on a subregion of the complex plane. In practice, we take the formal sum $\sum_{n=0}^{\infty} z^{n}$, realize it converges everywhere inside the unit circle to the value $\frac{1}{1-z}$, and analytically continue the sum to a meromorphic function on the complex plane. So, we pair $\sum_{n=0}^\infty 2^{n}$ with $\frac{1}{1-2}=-1$. Of course, this last calculation becomes a formal equality if we are working in the $s$-adic number system where $\sum_{n=0}^{N} s^{n}=\frac{s^{N+1}-1}{s-1}$ and $\lim_{N \to \infty}~ \sum_{n=0}^{N} s^{n}=\lim_{N \to \infty}~ \frac{s^{N+1}-1}{s-1}=\frac{-1}{s-1}$. Moreover, we interpret $\sum_{n=1}^\infty n $ to be the value of the analytic continuation of the Riemann zeta function at $s=-1$. Regularizing $\sum_{n=1}^\infty n$  yields the value $\zeta(-1)=-\frac{1}{12}$. The expression $\sum_{n=1}^\infty (n+1/2)$ is recognized to be the analytic continuation of the Hurwitz zeta function and evaluated along similar lines. 

While calculations of the Casimir Effect using zeta regularization have been experimentally verified in \cite{Bressi2002} and \cite{Sparnaay1957}, the results do not imply that that nature always chooses analytic regularization. Rather, we employ the method carefully in the interest of obtaining finite answers to questions which seem closely related to Casimir's original setup and in the hope that analogous reasoning applies here in the case of Laakso spaces. 
\subsection{Formulae for the Casimir Energy and Force on a General Laakso Space}
\label{sect:CasimirForce}
\begin{definition}
\label{SpectrumDef}
Let $\sigma(A) =\left\{ \lambda_{n} \right\}$ be the spectrum of the differential operator A with respective multiplicities $\left\{ g_{n} \right\} $. Then the spectral zeta function $\zeta(s)=~\sum_{n=1}^\infty \frac{g_{n}}{\lambda_{n}^{s}}$. 
\end{definition}
\begin{corollary}
\label{SpectralFunction}
Given a Laakso space configuration specified by $ j_{n}=N ~\forall n, X_{0} \in \left( 0, 1/2 \right), and~Z \in \mathbb{N}\cup{\left\{0 \right\} }$, the spectral zeta function of the self-adjoint operator $\Delta$ is
\begin{align*}
\zeta_{N,X_{0}, Z}(s) =&~  \sum _{k=0}^\infty \frac{2}{[(2k+1) \pi/(1-2X_{0})]^{2s}} +\sum_{k=1}^\infty \frac{1}{[k \pi/(2X_{0})]^{2s}}\\& +\sum _{k=1}^\infty \frac{(N-Z-3)}{[Nk \pi\frac{(1-(Z+1)/N)}{1-2X_{0}}]^{2s}}+\sum _{k=1}^\infty \frac{Z+1}{[k \pi(Z+1) / (2X_{0})]^{2s}} \\ 
& + \sum_{n=1}^\infty \sum_{k=0}^\infty \frac{2^{n}}{[I_{n}(k+1/2) \pi (1-\frac{Z+1}{N})/(1-2X_{0})]^{2s}} \\ & + \sum_{n=2}^\infty \sum_{k=1}^\infty \frac{(1-\frac{Z+1}{N})2^{n-1}I_{n-1}(N-2)+2^{n-1}(1-(Z+1)/N) I_{n-1}}{[I_{n}k \pi \frac{(1-(Z+1)/N)}{(1-2X_{0})}]^{2s}} \\& + \sum_{n=2}^\infty \sum_{k=1}^\infty \frac{ 2^{n-2}[(1-(Z+1)/N)I_{n-1}-1]-2^{n-2}}{ [ I_{n}k \pi(1-(Z+1)/N)/[2(1-2X_{0})]]^{2s}} \\  & +\sum_{n=2}^\infty \sum_{k=1}^\infty \frac{(Z+1)/N[2^{n-1}I_{n-1}(N-2)+2^{n-1}I_{n-1}]+2^{n-1}}{[k \pi I_{n}(Z+1)/(2NX_{0})]^{2s}} \\  & + \sum_{n=2}^\infty \sum_{k=1}^\infty  \frac{(Z+1)/N[2^{n-2}I_{n-1}]-2^{n-2}}{[I_{n}k \pi / (4NX_{0})]^{2s}}.
\end{align*}
 \end{corollary}
 \begin{proof}
 This follows immediately from an application of Definition \ref{SpectrumDef} to the results of Theorem \ref{Spectrum}.
 \end{proof}
 
 \begin{proposition}
 Let $\zeta_{N,X_{0}, Z}(s)$ be the spectral zeta function. Then $\frac{\hbar}{2}\zeta_{N,X_{0},Z}(-\frac{1}{2})$ gives the Casimir energy $E_{C}$ of the conducting plates in the Laakso space. 
 \end{proposition}
 
 \begin{proof}
 This follows from the fact that $E_{0}=\frac{\hbar}{2} \sum_{n} \omega_{n}$ and $\left\{ \omega_{n} \right\}=\left\{ \sqrt{ \lambda_{n}} ~|~ \lambda_{n} \in \sigma(\Delta) \right\}$. 
 \end{proof}
 
\begin{proposition}
\label{ForceProp}
Let two conducting metal plates be situated a distance $X_{0}$ symmetrically about x=$\frac{1}{2}$ in a Laakso space with $j_{n}=N~\forall n$. Furthermore, let Z nodes lie between the conductors in the $F_{1}$ quantum graph approximation. Then the Casimir Force $F_{C}$ experienced by each of the plates is given by $\frac{\hbar}{2} \frac{d}{dx} \zeta_{N,x,Z}(-\frac{1}{2}) |_{x=X_{0}}$ where a positive sign indicates an attractive force. 
 \end{proposition}
 
\begin{proof}
Because of the bilateral symmetry of the arrangement, the forces experienced by each of the plates must be equal in magnitude and opposite in direction.  The force experienced by a system is given by the negative energy gradient, so our expression for $F_{C}$ is correct up to sign. Lastly, if $\frac{d}{dx} \zeta_{N,x,Z}(-\frac{1}{2})|_{x=X_{0}}$ is positive, energy increases as the plates move apart so that it is energetically favorable for the plates to move closer together, which means the force is attractive as claimed.
\end{proof}

\begin{proposition}
The generalized Casimir Force experienced by two uncharged conducting metal plates is 
\begin{align*}
F_{C}&= \frac{2 \hbar \pi (N-(Z+1))}{24(1-2N)(1-2X_{0})^{2}} - \frac{\hbar (N-(Z+3))(N-(Z+1))}{12(1-2X_{0})^{2}} \\&- \frac{ 2 \hbar \pi N^{3} (N-2)}{12 (1-2N^{2})}\left\{ \frac{1-\frac{Z+1}{N}}{1-2X_{0}} \right\}^{2}
-\frac{5 \hbar \pi (1-\frac{Z+1}{N})}{24 (1-2X_{0})^{2}} \left\{ \frac{N^{2}(N-(Z+1)}{1-2N^{2}}-\frac{N^{2}}{1-2N} \right\} \\&+ \frac{ \hbar \pi (Z+1)^{2}}{48 X_{0}^{2}}
+ \frac{ \hbar \pi N (Z+1)^{2} (N-2)}{24 X_{0}^{2} (1-2N^{2})} + \frac{5 \hbar \pi N (Z+1)^{2}}{96 (1-2N^{2}) X_{0}^{2}} + \frac{ \hbar \pi}{6 (1-2X_{0})^{2}}\\& + \frac{ \hbar \pi}{48 X_{0}^{2}} - \frac{\hbar \pi N^{2} (1-\frac{Z+1}{N})}{24 (1-2N) (1-2X_{0})^{2}}+\frac{ \hbar \pi N (Z+1)}{96 X_{0}^{2}(1-2N)}-\frac{\hbar \pi N^{2} (1-\frac{Z+1}{N})}{12 (1-2X_{0})^{2}(1-2N)}\\& + \frac{\hbar \pi (Z+1) N}{48 X_{0}^{2} (1-2N)}.
\end{align*}
\end{proposition}
\begin{proof}
Using the expression in Corollary~\ref{SpectralFunction}, take the derivative as instructed in Proposition~\ref{ForceProp} and then reduce the result using the analytic continuation techniques discussed in Subsection \ref{Proviso}.  
\end{proof}

\subsection{The Spectral Zeta Function for Laakso Spaces} \label{sect: Zeta}

  \begin{theorem}\label{spectral}
For a Laakso space defined by a repeating sequence of $j_n$ with period $T$, the spectral zeta function can be analytically continued to the following function:
\begin{eqnarray}\label{spectralequation}\zeta_L(s)=&&\frac{\zeta_R(2s)}{\pi^{2s}}\left[\sum_{p=2}^{T+1} \left(\left(\frac{I_T^{2s}}{I_T^{2s}-I_T2^T}\right)\left(\frac{(2^{p-1})(I_{p-1})(2^{2s-1}+j_p-1)}{I_p^{2s}}\right)\right.\right.\nonumber\\
&&+\left.\left.\left(\frac{I_T^{2s}}{I_T^{2s}-2^T}\right)\left(\frac{(2^{p-1})(\frac{3}{2}2^{2s}-3)}{I_p^{2s}}\right)\right)+\frac{2^{2s+1}-4+j_1}{j_1^{2s}}+1\right].
\end{eqnarray}
\end{theorem}

\begin{proof} The following is from \cite[Chapter 6]{ST08}: 
\begin{eqnarray}
\zeta_L(s)&=&\frac{\zeta_R(2s)}{\pi^{2s}}\left[\left(\sum_{n=2}^{\infty}\frac{2^{n-1}(I_{n-1})(2^{2s-1}+j_n-1)+2^{n-1}(\frac{3}{2}2^{2s}-3)}{I_n^{2s}}\right)\right.\nonumber \\
&&+\left.\frac{2^{2s+1}-4+j_1}{j_1^{2s}}+1\right].
\end{eqnarray}
\noindent gives us the value of the spectral zeta function for a generalized Laakso space.  Since we have the sequence of $j_n$ repeating with period \emph{T}, we know that for any nonnegative integer $n$, and for any integer $p$, where $0\leq{p}<T$, $I_{p+nT}=(I_T^n)(I_p)$.  For $s>1$, the series is absolutely convergent, so we can split up the terms by their remainders (mod $k$) and write the series as follows:
\begin{eqnarray}\label{spectral1}
\zeta_L(s)&=&\frac{\zeta_R(2s)}{\pi^{2s}}\left[\sum_{p=2}^{T+1}\left(\sum_{n=0}^{\infty}\left(\frac{2^{nT+p-1}(I_{nT+p-1})(2^{2s-1}+j_{nT+p}-1)}{I_{nT+p}^{2s}}+\right.\right.\right.\nonumber\\
&&\left.\left.\left.\frac{2^{nT+p-1}(\frac{3}{2}2^{2s}-3)}{I_{nT+p}^{2s}}\right)\right)+\frac{2^{2s+1}-4+j_1}{j_1^{2s}}+1\right].
\end{eqnarray}
Since $I_{p+nT}=(I_T^n)(I_p)$ and $j_{p+nT}=j_p$,  we are left with a finite sum of geometric series which are absolutely convergent when $Re(s) > 1$ and meromorphically extendable to the rest of the complex plane.
\end{proof}

\begin{corollary}\label{1/2}
For every Laakso space except the space where $j_n=2$ for all $n$,

\begin{eqnarray}
\lim_{s \rightarrow \frac{1}{2}} \zeta_L(s)&=&\frac{1}{2\pi}\left[\sum_{p=2}^{T+1}\left(\frac{2^p\ln2}{j_p(1-2^T)}-\frac{2^p\ln(I_p)}{1-2^T}+\frac{2^pI_T(3\ln2)}{I_p(I_T-2^T)}\right)\right.\\
&&\left.+\frac{2^{T+2}\ln(I_T)}{1-2^T}+\frac{8\ln2}{j_1}-2\ln(j_1)\right].\nonumber
\end{eqnarray}
\end{corollary}

\begin{proof}
This follows from Theorem \ref{spectral}, using L'Hospital's rule, since the Riemann zeta function has a simple pole at 1.
\end{proof}

\begin{corollary}
The spectral zeta function of a repeating Laakso space has poles at
\begin{equation}\label{poles}
\bigcup_{m \in \mathbb{Z}}\left\{\frac{\ln (2^TI_T) +2T\pi im}{\ln (I_T^2)}\right\} \cup \bigcup_{m \in \mathbb{Z}}\left\{\frac{\ln (2^T) +2T\pi im}{\ln (I_T^2)}\right\}.\nonumber
\end{equation}
\end{corollary}

\begin{proof}
This follows from Theorem \ref{spectral} and Corollary \ref{1/2} and is consistent with the results in \cite{ST08}.
 \end{proof}

\begin{proposition}
The spectral dimension of the Laakso space with period T is
\begin{equation}\label{eqdim}
d_s=\frac{\ln\left(2^TI_T\right)}{\ln\left(I_T\right)}.
\end{equation}
\end{proposition}

\begin{proof}
In the periodic case we have a closed form for the meromorphic spectral zeta function. As in \cite{ST08} the spectral dimension is taken to be the largest real part of the poles of the zeta function. The proposition follows by inspection.
\end{proof}

From Theorem \ref{spectral} we can calculate the value of the spectral zeta function at specific values of $j$, where  $j_n=j$ for all values of $n$. 

\begin{corollary}\label{constantj}
The following expressions give values for the spectral zeta function on a Laakso space with constant $j_n$:
{\small
\begin{eqnarray}\label{>1/2}
\label{<1/2}\zeta_L(s)&=&\frac{\zeta_R(2s)}{\pi^{2s}}\left[\frac{j^{4s}-j^{2s+1}+2(2^{2s}j^{2s})-6j^{2s}-3(2^{2s}j)+8j-2^{2s}+2}{(j^{2s}-2j)(j^{2s}-2)}\right].
\end{eqnarray}
}
\end{corollary}

Notably, we have the following value for $\zeta_L(-\frac{1}{2})$:

\begin{eqnarray}\zeta_L\left(-\frac{1}{2}\right)=\frac{-\pi}{12}\left(\frac{13}{8}+\frac{3j-2}{8j^2-4}+\frac{9}{16j-8}\right)
\end{eqnarray}

\begin{corollary}\label{alternatingj}

The spectral zeta function for a sequence of $j_n$ with period $2$, is

\begin{eqnarray}\zeta_L&=&\frac{\zeta_R(2s)}{\pi^{2s}}\left[\left(\frac{2j_1}{I_2^{2s}-4I_2}\right)\left(2^{2s-1}+j_2-1+\frac{2j_2(2^{2s-1}+j_1-1)}{j_1^{2s}}\right)\right.\\
&&+\left.\left(\frac{3(2^{2s})-6}{I_2^{2s}-4}\right)\left(1+\frac{2}{j_1^{2s}}\right)+\frac{2^{2s+1}-4+j_1}{j_1^{2s}}+1\right].\nonumber
\end{eqnarray}
\end{corollary}

\subsection*{Acknowledgments} 

The authors would like to acknowledge Matthew Begue, Levi DeValve, David Miller and Kevin Romeo for making available their MatLab code which served as the basis for calculations presented in Section \ref{sect: SA}.

%\nocite{*}
\bibliography{Bibliography}{}
\bibliographystyle{plain}

\end{document}